\documentclass[12pt,A4]{amsart}
\usepackage{amsfonts,amsthm,amsmath}
\usepackage[mathscr]{eucal}
\usepackage{cite}
\usepackage{epic}
\usepackage{hyperref}

\usepackage{xcolor}
\usepackage{tikz}
\usepackage{tkz-graph}
\usepackage{collectbox}

\usetikzlibrary{shapes,arrows}
\tikzstyle{pnt}=[draw,circle,fill,inner sep=1pt]
\tikzstyle{opnt}=[draw,circle,inner sep=2pt]

\numberwithin{equation}{section}
\def\blue{\textcolor{blue}}
\def\red{\textcolor{red}}

\newtheorem{theorem}{Theorem}[section]
\newtheorem{lemma}[theorem]{Lemma}

\newtheorem{proposition}[theorem]{Proposition}
\newtheorem{Fact}[theorem]{Fact}
\newtheorem{Def}[theorem]{Definition}

\theoremstyle{definition}
\newtheorem{example}[theorem]{Example}

\newtheorem{Prob}[theorem]{Problem}
\newcommand\bi[2]{{{#1}\atopwithdelims(){#2}}}

\newcommand{\N}{\mathbb{N}}
\def\Z{\mathbb{Z}}
\def\NB{\mathrm{BNW}}
\def\NW{\mathrm{NW}}
\def\NC{\mathrm{NC}}
\def\Com{\mathrm{Com}}
\def\arc{\mathrm{arc}}
\def\one{\mathrm{one}}
\def\redd{\mathrm{red}}
\def\graphsize{.9}

\def\boxit#1{\leavevmode\hbox{\vrule\vtop{\vbox{\kern.33333pt\hrule\kern1pt\hbox
{\kern1pt\vbox{#1}\kern1pt}}\kern1pt\hrule}\vrule}}



\begin{document}
\title{A combinatorial bijection on $k$-noncrossing partitions}
\author[Z. Lin]{Zhicong Lin}
\address[Zhicong Lin]{Research Center for Mathematics and Interdisciplinary Sciences, Shandong University, Qingdao 266237, P.R. China}
\email{zhicong.lin@univie.ac.at}

\author[D. Kim]{Dongsu Kim}
\address[Dongsu Kim]{Department of Mathematical Sciences, Korea Advanced Institute of Science
and Technology, Daejeon 34141, Republic of Korea} 
\email{dongsu.kim@kaist.ac.kr}
\date{\today}

\begin{abstract} 
For any integer $k\geq2$, we prove combinatorially the following Euler (binomial) transformation identity
$$
\NC_{n+1}^{(k)}(t)=t\sum_{i=0}^n{n\choose i}\NW_{i}^{(k)}(t),
$$
where $\NC_{m}^{(k)}(t)$ (resp.~$\NW_{m}^{(k)}(t)$) is the sum of weights, $t^\text{number of blocks}$,
of partitions of $\{1,\ldots,m\}$ without $k$-crossings (resp.~enhanced $k$-crossings).
The special $k=2$ and $t=1$ case, asserting the Euler transformation of Motzkin numbers are
Catalan numbers, was discovered by Donaghey 1977.
The result for $k=3$ and $t=1$, arising naturally in a recent study of pattern avoidance 
in ascent sequences and inversion sequences, was proved only analytically. 
\end{abstract}

\maketitle

\section{Introduction}

There has been recent interest in the distributions of $k$-crossings ($k$-nestings) and 
enhanced $k$-crossings (enhanced $k$-nestings) on set partitions. Using the RSK-like 
insertion/deletion algorithms, Chen, Deng, Du, Stanley and Yan~\cite{chen} developed two 
fundamental bijections between partitions and vacillating or hesitating tableaux, from 
which the symmetry of the joint distributions of $k$-crossings (resp.~enhanced
$k$-crossings) and $k$-nestings (resp.~enhanced $k$-nesting) follows. Their results were 
successfully put in a larger context of fillings of Ferrers shape by
Krattenthaler~\cite{kr0} via the growth diagram construction of Fomin. As applications of 
the two bijections of Chen et al., Bousquet-M\'elou and Xin~\cite{bx} enumerated set 
partitions avoiding usual or enhanced $3$-crossings. They also made the conjecture that 
for every $k>3$, the generating function of $k$-noncrossing partitions is not D-finite, 
which is still open. Recently, Burrill, Elizalde, Mishna and Yen~\cite{bemy} carried out
a different approach to the usual and enhanced $k$-nonnesting partitions through generating 
trees and open arc diagrams of partitions. 

In this paper, we prove combinatorially the following Euler (binomial) transformation 
identity, which indicates the numbers of usual and enhanced $k$-noncrossing partitions are 
closely related. 

\begin{theorem}\label{thm:main} Let $\NC_{m}^{(k)}(t)$ (resp.~$\NW_{m}^{(k)}(t)$) be the sum of
weights, $t^\text{number of blocks}$, of partitions of $[m]:=\{1,\ldots,m\}$ without $k$-crossings
(resp.~enhanced $k$-crossings). For $n\geq1$ and $k\geq2$,
\begin{equation}\label{eq:NCNW}
\NC_{n+1}^{(k)}(t)=t\sum_{i=0}^n{n\choose i}\NW_{i}^{(k)}(t),
\end{equation}
where $\NW_0^{(k)}(t)=1$ by convention. 
\end{theorem}
The $t=1$ case of~\eqref{eq:NCNW} implies that the $D$-finiteness
(see~\cite[Theorem 6.4.10]{st2}) of the generating function of $k$-noncrossing partitions 
is the same as that of the generating function of enhanced $k$-noncrossing partitions. 
There are several partial results that lead to the discovery of~\eqref{eq:NCNW}.
Before we state our motivation, some definitions on set partitions are required.
Let $\Pi_{n}$ be the set of all partitions of $[n]$.
Any $P\in\Pi_n$ can be identified with its {\em arc diagram} defined as follows:

\begin{Def}[Arc diagram of a partition]\label{ArcDiagram}
Nodes are $1,2,\dots,n$ from left to right. There is an arc from $i$ to $j$, $i<j$, whenever
both $i$ and $j$ belong to the same block, say $B\in P$, and $B$ contains no $l$ with $i<l<j$.
There is a loop from $i$ to itself if $\{i\}$ is a block in $P$.
\end{Def}

\begin{example}
The arc diagram of $\{\{1,3\},\{2,5,6\},\{4\}\}\in\Pi_6$:
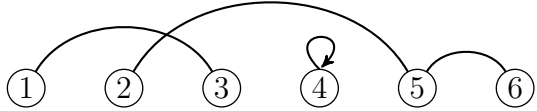
\begin{center}
\begin{tikzpicture}
\SetVertexNormal
\SetGraphUnit{1.3}
\tikzset{VertexStyle/.append style={inner sep=0pt,minimum size=5mm}}
\Vertex{1}
\EA(1){2}\EA(2){3}\EA(3){4}\EA(4){5}\EA(5){6}
\tikzset{EdgeStyle/.append style = {bend left = 60}}
\Loop[dist=.8cm, dir=EA](4.north)
\Edge(1)(3)\Edge(2)(5)
\Edge(5)(6)
\end{tikzpicture}
\end{center}
\end{example}

A partition has a {\em crossing} (resp.{\em~nesting}) if there exist two arcs
$(i_1,j_1)$ and $(i_2,j_2)$ in its arc diagram such that $i_1<i_2<j_1<j_2$
(resp.~$i_1<i_2<j_2<j_1$). It is well known (cf.~\cite{pe2}) that the number of
partitions in $\Pi_n$ with no crossings (or with no nestings) is given by the $n$-th
{\em Catalan number}
$$
C_n=\frac{1}{n+1}{2n\choose n}.
$$
The crossings (resp.~nestings) of partitions have a natural generalization called
$k$-crossings (resp.~$k$-nestings) for any fixed integer $k\geq2$.

\begin{Def}[$k$-crossing and enhanced $k$-crossing]
A {\em$k$-crossing} of $P\in\Pi_n$ is a $k$-subset $(i_1,j_1),(i_2,j_2),\ldots,(i_k,j_k)$ 
of arcs in the arc diagram of $P$ such that 
$$
i_1<i_2<\cdots<i_k<j_1<j_2<\cdots<j_k.
$$
A {\em weak $k$-crossing} of $P\in\Pi_n$ is a $k$-subset
$(i_1,j_1),(i_2,j_2),\ldots,(i_k,j_k)$ of arcs in the arc diagram of $P$ such that 
$$
i_1<i_2<\cdots<\blue{i_k= j_1}<j_2<\cdots<j_k.
$$
The $k$-crossings and the weak $k$-crossings of $P$ are collectively called the
{\em enhanced $k$-crossings} of $P$. A partition without any $k$-crossings
(resp.~enhanced $k$-crossings) is called {\em$k$-noncrossing}
(resp.{\em~enhanced $k$-noncrossing}). 
The {\em$k$-nestings} and the {\em enhanced $k$-nestings} of a partition are defined
similarly by replacing the above two inequalities with
$$
i_1<i_2<\cdots<i_k<j_k<j_{k-1}<\cdots<j_1\quad\text{and}
\quad i_1<i_2<\cdots<\blue{i_k= j_k}<j_{k-1}<\cdots<j_1.
$$
\end{Def}

\begin{example}
A $3$-crossing and a weak $3$-crossing are depicted below:
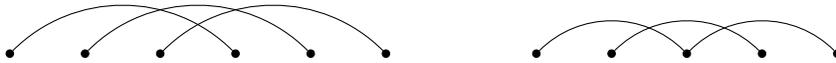
\begin{center}
\begin{tikzpicture}[scale=1]
\foreach \i in {1,...,6} \node[pnt,label=below:] at (\i,0)(\i) {};
\draw(1) to [bend left=45] (4);
\draw(2) to [bend left=45] (5);
\draw(3) to [bend left=45] (6);
\foreach \i in {8,...,12} \node[pnt,label=below:] at (\i,0)(\i) {};
\draw(8) to [bend left=45] (10);
\draw(9) to [bend left=45] (11);
\draw(10) to [bend left=45] (12);
\end{tikzpicture}
\end{center}
\end{example}
 
Note that enhanced $2$-noncrossing partitions in $\Pi_n$ are {\em noncrossing partial 
matchings} of~$[n]$, i.e.~noncrossing partitions for which the blocks have size one or 
two. Since noncrossing partial matchings of $[n]$ are counted by the $n$-th
{\em Motzkin number} $M_n=\sum_{i=0}^{\lfloor n/2\rfloor}{n\choose 2i}C_i$
(see~\cite[Page~43]{pe2}), identity~\eqref{eq:NCNW} reduces to 
\begin{equation}\label{eq:CM}
C_{n+1}=\sum_{i=0}^n{n\choose i}M_i
\end{equation}
when $k=2$ and $t=1$. This Motzkin--Catalan identity was first discovered by
Donaghey~\cite{do}, who interpreted it as a generating function counting plane trees
by number of branches. Many other interpretations in terms of different models are found later
in~\cite{doS,DS,deng} and lately in~\cite{ge}. Let $\NC_n^{(k)}$ (resp.~$\NW_n^{(k)}$) be 
the set of all $k$-noncrossing (resp.~enhanced $k$-noncrossing) partitions in~$\Pi_n$.
If $k$ is sufficiently large, i.e. $k>\frac{n+1}{2}$, then we have
$\NW_n^{(k)}=\NC_n^{(k)}=\Pi_n$, and (\ref{eq:NCNW}) is equivalent to
$$
\text{for all $m\geq0$, }\ S(n+1,m+1)=\sum_{i=0}^n\bi{n}{i}S(i,m)
$$
where $S(a,b)$ denotes the {\em Stirling number of the second kind}.

The first not so trivial case of~\eqref{eq:NCNW} is when $k=3$ and $t=1$, which
arises naturally in establishing an equinumerosity conjecture of Yan~\cite{yan} and 
Martinez--Savage~\cite{ms}.
It was proved in~\cite{lin} analytically that {\em inversion sequences} with no weakly
decreasing subsequence of length $3$ and enhanced $3$-noncrossing partitions have the 
same cardinality, from which, together with the connections with $021$-avoiding
{\em ascent sequences}, this special case of~\eqref{eq:NCNW} follows.
Moreover, the $t=1$ case of~\eqref{eq:NCNW} for general $k$ was also conjectured
in~\cite{lin}. On the other hand, even the $k=2$ case of~\eqref{eq:NCNW}, which is a
$t$-extension of~\eqref{eq:CM}, seems new:
\begin{equation}\label{eq:t-CM}
C_{n+1}(t)=t\sum_{i=0}^n\bi{n}{i}M_i(t).
\end{equation}
Here $C_{n}(t)$ and $M_n(t)$ denote respectively the generating functions of noncrossing 
partitions of $[n]$ and noncrossing partial matchings of $[n]$, where $t$ corresponds to 
a block in partitions or matchings. Note that $C_{n}(t)$ is the $n$-th
{\em Narayana polynomial} $\sum_{k=1}^n\frac{1}{n}{n\choose k}{n\choose k-1}t^k$, while 
$M_n(t)$ turns out to be the {\em$\gamma$-polynomial} of $C_{n+1}(t)$ as we shall see in 
Section~\ref{cross1}. 

It should be pointed out that Theorem~\ref{thm:main} is not true (except when $t=1$)
if one replaces ``enhanced $k$-noncrossing partitions'' by ``enhanced $k$-nonnesting 
partitions'', as distributions of the number of blocks on these two objects are different 
(see also~\cite[Theorem~4.3]{chen}). 

The rest of this paper is organized as follows. In Section~\ref{trian}, we will show how 
the $t=1$ case of~\eqref{eq:NCNW} and Eq.~\eqref{eq:t-CM} follow easily from two 
different representations of set partitions as $01$-fillings of a triangular shape. The 
complete bijective proof of~\eqref{eq:NCNW} is provided in Section~\ref{NCNW} through the 
introduction of two well-designed operations on $k$-noncrossing partitions. We conclude 
our paper with two related open problems in Section~\ref{sec:final}.

\section{Via $01$-filling of triangular shape}
\label{trian}
We shall review two representations of set partitions as $01$-fillings of a triangular 
shape from~\cite{kr0} (see also~\cite{yan2}), from which one can see the special $t=1$ 
case of~\eqref{eq:NCNW} directly.

A {\em triangular shape} of order $n$, denoted by $\triangle_n$, is the Young diagram of 
the staircase partition $(1,2,\ldots,n)$ in the French notation. For a triangular shape, we 
number its rows from top to bottom and its columns from left to right and identify cells 
using matrix coordinates. 
For example, the triangular shape $\triangle_5$ is depicted below with rows and columns 
numbered (the cell $(4,2)$ is colored blue):
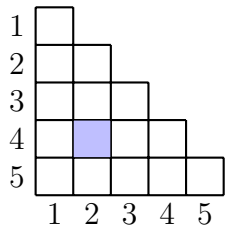
\begin{center}
\begin{tikzpicture}[scale=.5]
\draw [thick]
(0,0)--(5,0) (0,1)--(5,1) (0,2)--(4,2) (0,3)--(3,3)(0,4)--(2,4) (0,5)--(1,5)
(0,0)--(0,5)(1,0)--(1,5) (2,0)--(2,4)(3,0)--(3,3) (4,0)--(4,2) (5,0)--(5,1);
\draw(-0.5,4.5) node{$1$};
\draw(-0.5,3.5) node{$2$};
\draw(-0.5,2.5) node{$3$};
\draw(-0.5,1.5) node{$4$};
\draw(-0.5,0.5) node{$5$};

\draw(0.5,-0.5) node{$1$};
\draw(1.5,-0.5) node{$2$};
\draw(2.5,-0.5) node{$3$};
\draw(3.5,-0.5) node{$4$};
\draw(4.5,-0.5) node{$5$};
\draw[fill=blue!25] (1,1) rectangle (2,2);

\end{tikzpicture}
\end{center}

A triangular shape $\triangle_n$ with each cell filled by either a $0$ or a $1$ is called 
a {\em $01$-filling} of $\triangle_n$. For better visibility, we will represent
a $01$-filling by replacing a $1$ with a $\bullet$ but suppressing all occurrences of $0$. 
An {\em SE-chain} ({\em South-East chain}) of a $01$-filling of $\triangle_n$ is a set of 
$1$'s such that any $1$ in the sequence is below and to the right of the preceding $1$ in 
the sequence. Moreover, an SE-chain is said to be {\em proper} if the smallest rectangle 
containing the chain is contained in $\triangle_n$. For instance, the $01$-filling in the left 
side of Fig.~\ref{filling} has two SE-chains of length $3$, one chain containing $1$'s in 
cells $(4,1), (6,3), (7,5)$ and the other chain containing $1$'s in cells
$(4,1), (6,3), (8,4)$, where only the latter chain is proper. 

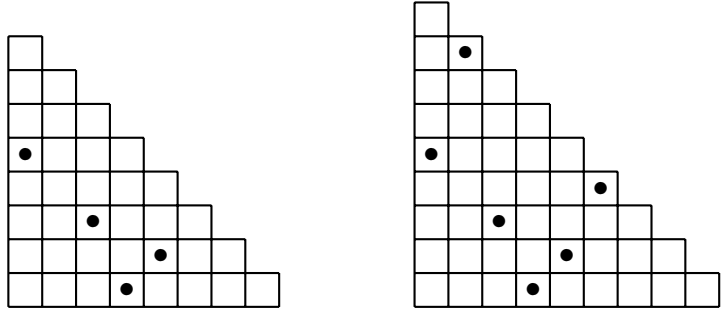
\begin{figure}
\begin{center}
\begin{tikzpicture}[scale=.45]
\draw [thick]
(0,0)--(8,0) (0,1)--(8,1) (0,2)--(7,2) (0,3)--(6,3) (0,4)--(5,4) (0,5)--(4,5) (0,6)--(3,6) (0,7)--(2,7) (0,8)--(1,8)

(0,0)--(0,8)(1,0)--(1,8) (2,0)--(2,7)(3,0)--(3,6) (4,0)--(4,5) (5,0)--(5,4) (6,0)--(6,3) (7,0)--(7,2) (8,0)--(8,1);

\draw(0.5,4.5) node{$\bullet$};\draw(2.5,2.5) node{$\bullet$};\draw(3.5,0.5) node{$\bullet$};
\draw(4.5,1.5) node{$\bullet$};
\draw [thick]
(12,0)--(21,0) (12,1)--(21,1) (12,2)--(20,2) (12,3)--(19,3) (12,4)--(18,4) (12,5)--(17,5) 
(12,6)--(16,6) (12,7)--(15,7) (12,8)--(14,8) (12,9)--(13,9)

(12,0)--(12,9) (13,0)--(13,9) (14,0)--(14,8) (15,0)--(15,7) (16,0)--(16,6) (17,0)--(17,5) (18,0)--(18,4)
(19,0)--(19,3) (20,0)--(20,2) (21,0)--(21,1);

\draw(12.5,4.5) node{$\bullet$};\draw(13.5,7.5) node{$\bullet$};\draw(14.5,2.5) node{$\bullet$};
\draw(15.5,0.5) node{$\bullet$}; \draw(16.5,1.5) node{$\bullet$};\draw(17.5,3.5) node{$\bullet$};
\end{tikzpicture}
\end{center}
\caption{Two $01$-filling representations $\mathcal{C}(P)$ and $\mathcal{E}(P)$ for
the set partition $P=\{\{1,5,8\},\{2\},\{3,7\},\{4,9\},\{6\}\}$.\label{filling}}
\end{figure}

A set partition $P\in\Pi_n$ can be encoded by a $01$-filling of $\triangle_{n-1}$ 
(resp.~of $\triangle_{n}$), denoted $\mathcal{C}(P)$ (resp.~$\mathcal{E}(P)$), which is 
the $01$-filling with a $1$ in cell $(j-1,i)$ (resp.~$(j,i)$) if and only if $(i,j)$ with $i<j$
(resp. $i\leq j$) is an arc in the arc diagram of $P$. See Fig.~\ref{filling} for the $01$-filling 
representations $\mathcal{C}$ and $\mathcal{E}$ of partition
$\{\{1,5,8\},\{2\},\{3,7\},\{4,9\},\{6\}\}\in\Pi_9$.
Denote by $\mathcal{C}_n^{(k)}$ the set of $01$-fillings $F$
of $\triangle_{n}$ satisfying three conditions:
(1) each row in $F$ contains at most one $1$,
(2) each column in $F$ contains at most one $1$,
and (3) $F$ has no proper SE-chain of length $k$.
Let $\mathcal{E}_n^{(k)}$ denote the subset of $\mathcal{C}_n^{(k)}$ consisting of all
elements $F$ such that, for each $i$, either the $i$-th row or the $i$-th column of
$F$ contains a $1$.
The following fact, first realized by Krattenthaler~\cite{kr0}, is the crucial property
of $\mathcal{C}$ and~$\mathcal{E}$.
\begin{theorem}\label{th:krat}
The mappings $\mathcal{C}:\NC_{n+1}^{(k)}\rightarrow\mathcal{C}_n^{(k)}$ and
$\mathcal{E}:\NW_n^{(k)}\rightarrow\mathcal{E}_n^{(k)}$ are bijections.
\end{theorem}

For each $j\in[n]$, the hook of $\triangle_n$ formed by the $j$-th row and the $j$-th 
column is called the {\em $j$-th corner hook} of $\triangle_n$. A corner hook of a
$01$-filling is said to be {\em zero}, if it does not contain any $1$; otherwise, {\em nonzero}.
It is clear that $\mathcal{E}_n^{(k)}$ is the set of all $01$-fillings in
$\mathcal{C}_n^{(k)}$ with no zero corner hooks. This inspires a natural bijection $f$ 
from $\mathcal{C}_n^{(k)}$ to the disjoint union of the cartesian products of
$(i+1)$-compositions of $n+1$ and $\mathcal{E}_i^{(k)}$ for $i=0,1,\ldots,n$.

Recall that an {\em $i$-composition of $n$} is a sequence $c_1c_2\cdots c_i$ of positive 
integers summing to~$n$. Let $\Com_i(n)$ denote the set of all
$i$-compositions of $n$. For each $01$-filling $F$ in $\mathcal{C}_n^{(k)}$ we set
$f(F)=(c,\mathcal{D}(F))\in\Com_{i+1}(n+1)\times\mathcal{E}_i^{(k)}$, where
$c=c_1c_2\cdots c_{i+1}\in\Com_{i+1}(n+1)$,
$$
\{c_1,c_1+c_2,\dots,c_1+c_2+\cdots+c_i\}=\{j\in[n]:\text{the $j$-th corner hook of $F$
is nonzero}\},
$$
and $\mathcal{D}(F)$ is obtained from $F$ by removing all the cells in zero corner 
hooks and compressing the remaining cells. 

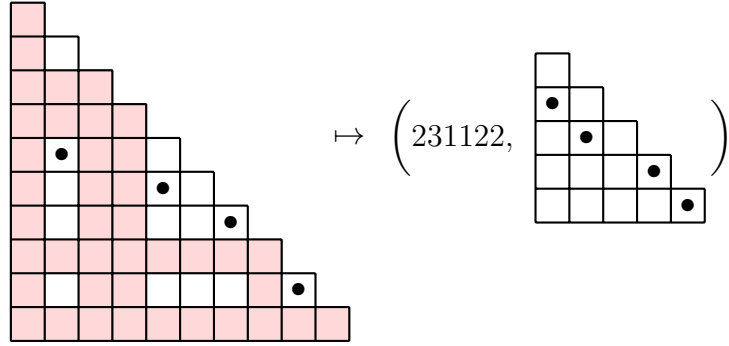
\begin{figure}
\begin{center}
\begin{tikzpicture}[scale=.45]
\draw[fill=red!15] (0,0) rectangle (10,1);
\draw[fill=red!15] (0,2) rectangle (8,3);
\draw[fill=red!15] (0,6) rectangle (4,7);
\draw[fill=red!15] (0,7) rectangle (3,8);

\draw[fill=red!15] (0,0) rectangle (1,10);
\draw[fill=red!15] (2,0) rectangle (3,8);
\draw[fill=red!15] (3,0) rectangle (4,7);
\draw[fill=red!15] (7,0) rectangle (8,3);
\draw [thick]
(0,0)--(10,0) (0,1)--(10,1) (0,2)--(9,2) (0,3)--(8,3) (0,4)--(7,4) (0,5)--(6,5) (0,6)--(5,6) (0,7)--(4,7) (0,8)--(3,8) (0,9)--(2,9) (0,10)--(1,10)

(0,0)--(0,10)(1,0)--(1,10) (2,0)--(2,9)(3,0)--(3,8) (4,0)--(4,7) (5,0)--(5,6) (6,0)--(6,5) (7,0)--(7,4) (8,0)--(8,3) (9,0)--(9,2) (10,0)--(10,1);

\draw(1.5,5.5) node{$\bullet$};\draw(4.5,4.5) node{$\bullet$};\draw(6.5,3.5) node{$\bullet$};
\draw(8.5,1.5) node{$\bullet$};

\draw(10,6) node{$\mapsto$};
\draw(13,6) node{$\biggl(231122,$};
\draw(21,6) node{$\biggr)$};
\draw [thick]
(15.5,3.5)--(20.5,3.5) (15.5,4.5)--(20.5,4.5) (15.5,5.5)--(19.5,5.5) 
(15.5,6.5)--(18.5,6.5) (15.5,7.5)--(17.5,7.5)(15.5,8.5)--(16.5,8.5)

(15.5,3.5)--(15.5,8.5) (16.5,3.5)--(16.5,8.5) (17.5,3.5)--(17.5,7.5) (18.5,3.5)--(18.5,6.5) (19.5,3.5)--(19.5,5.5) (20.5,3.5)--(20.5,4.5);

\draw(16,7) node{$\bullet$};\draw(17,6) node{$\bullet$};\draw(19,5) node{$\bullet$};
\draw(20,4) node{$\bullet$}; 
\end{tikzpicture}
\end{center}
\caption{An example of $f$.\label{del}}
\end{figure}
See Fig.~\ref{del} for an example of $f$. It is easy to see that $f$ is invertible.

\begin{theorem}\label{th:f}
The mapping $f: \mathcal{C}_n^{(k)}\rightarrow\bigcup_{i=0}^n\Com_{i+1}(n+1)\times\mathcal{E}_i^{(k)}$
is a bijection. 
\end{theorem}

Since the cardinality of $\Com_{i+1}(n+1)$ is ${n\choose i}$, Theorem~\ref{th:f} and 
Theorem~\ref{th:krat} together provide a combinatorial proof of~\eqref{eq:NCNW} for $t=1$.
\subsection{On identity~\eqref{eq:t-CM} and $\gamma$-expansion of Narayana polynomials}
\label{cross1}

A polynomial $H(t)=\sum_{j=1}^dh_jt^j$ is said to be {\em palindromic} if $h_j=h_{d-j}$ 
for all $j=1,\ldots,d-1$. Any palindromic polynomial $H(t)\in\Z[t]$ can be written 
uniquely as
$$
H(t)=\sum_{i=1}^{\lfloor\frac{d+1}{2}\rfloor}\gamma_i\,t^i(1+t)^{d+1-2i},
$$
where $\gamma_i\in\Z$. If $\gamma_i\geq0$ for all $i$, then $H(t)$ is said to be
{\em $\gamma$-positive}. It is well known and not difficult to see that
$\gamma$-positivity implies unimodality, namely 
$$
h_1\leq h_2\leq\cdots\leq h_{\lfloor(d+1)/2\rfloor}\geq\cdots\geq h_{d-1}\geq h_d. 
$$
Many polynomials arising from algebraic combinatorics and discrete geometry have been 
shown to be $\gamma$-positive. The reader is referred to the book exposition by 
Petersen~\cite{pe2} for further information on $\gamma$-positivity. 

Note that Narayana polynomials $C_n(t)$ are $\gamma$-positive. For instance, $C_5(t)$ is
$\gamma$-positive since
$$
C_5(t)=t+10t^2+20t^3+10t^4+t^5=t(1+t)^4+6t^2(1+t)^2+2t^3.
$$
In general, we have the following $\gamma$-expansion of $C_{n+1}(t)$
(cf.~\cite{bp} or~\cite{prw}):
$$
C_{n+1}(t)=t\sum_{i=0}^{\lfloor\frac{n}{2}\rfloor}\gamma_{n+1,i}\,t^i(1+t)^{n-2i},
$$
where $\gamma_{n+1,i}={n\choose 2i}C_i$. 
More interestingly, the coefficient $\gamma_{n+1,i}$ has several combinatorial 
interpretations, one of which shows $\gamma_{n+1,i}$ is the number of Motzkin paths of 
length $n$ with $i$ up steps (cf.~\cite[Lemma~18]{lin0}). It is somewhat surprising that 
$\gamma_{n+1,i}$ appears as a coefficient of $M_{n}(t)$, even though the proof of this 
relationship is simple. 

\begin{proposition}\label{gamma:match}
The number of noncrossing partial matchings of $[n]$ with $i$ blocks
is~$\gamma_{n+1,n-i}$. In other words,
$$
M_n(t)=\sum_{i=\lfloor(n+1)/2\rfloor}^n\gamma_{n+1,n-i}t^i.
$$
\end{proposition}
\begin{proof}
We will apply a classical bijection between noncrossing partial matchings of $[n]$ and 
Motzkin paths of length $n$. 

Recall that a {\em Motzkin path} of length $n$ is a lattice path in $\N^2$ starting at
$(0,0)$, ending at $(n,0)$, with three possible steps: 
$$
\text{$(1,1)=U$ (up step), $(1,-1)=D$ (down step),\,\,\,and\,\,\,$(1,0)=H$
(horizontal step).}
$$
For a noncrossing partial matching $P\in\NW_n^{(2)}$, define the corresponding Motzkin 
path $M=s_1s_2\ldots s_n$, represented as a word on the alphabet $\{U,D,H\}$, by
$$
s_{i}=
\begin{cases}
\,U,\quad&\text{if $(i,b)$ is an arc in the arc diagram of $P$},\\
\,D,\quad&\text{if $(a,i)$ is an arc in the arc diagram of $P$},\\
\,H,\quad&\text{if $(i,i)$ is a loop in the arc diagram of $P$}.
\end{cases}
$$
For an example of this correspondence, see Fig.~\ref{fig:motzkin}. 
\begin{figure}
\begin{center}
\begin{tikzpicture}[scale=.8]
\draw[step=1,color=gray] (0,0) grid (12,3); \draw [very thick]
(0,0)--(1,1) --(2,2)--(3,1)--(4,2)--(5,3)--(6,2)--(7,2)--(8,1)--(9,0)--(10,1)--(11,1)--(12,0);
\draw (0.5,-0.4) node{$1$};\draw (1.5,-0.4) node{$2$};
\draw (2.5,-0.4) node{$3$};\draw (3.5,-0.4) node{$4$};
\draw (4.5,-0.4) node{$5$};\draw (5.5,-0.4) node{$6$};
\draw (6.5,-0.4) node{$7$};\draw (7.5,-0.4) node{$8$};
\draw (8.5,-0.4) node{$9$};\draw (9.5,-0.4) node{$10$};
\draw (10.5,-0.4) node{$11$};\draw (11.5,-0.4) node{$12$};
\end{tikzpicture}
\end{center}
\caption{The Motzkin path corresponds to noncrossing partial matching
$\{\{1,9\},\{2,3\},\{4,8\},\{5,6\},\{7\},\{10,12\},\{11\}\}$. }
\label{fig:motzkin}
\end{figure}
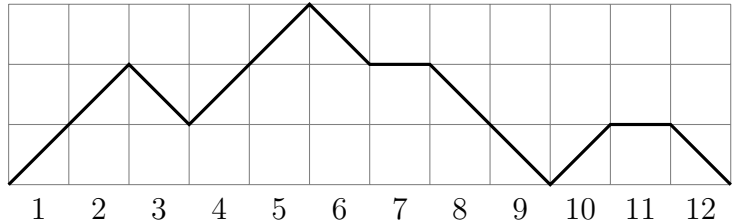
This correspondence is one-to-one and maps the number of loops and the number
of non-loop arcs of $P$ to the number of horizontal steps and the number of up steps of $M$,
respectively. The result then follows from the fact that $\gamma_{n+1,n-i}$ is the number of
Motzkin paths of length $n$ with $n-i$ up steps. 
\end{proof}

It is possible to deduce identity~\eqref{eq:t-CM} from Theorem~\ref{th:f} and the 
palindromity of the Narayana polynomials. Let $\arc(P)$ (resp.~$\arc_{2}(P)$) be the 
number of arcs (resp.~non-loop arcs) in the arc diagram of partition $P$. It is clear 
that $|P|=n-\arc_2(P)$ for any $P\in\Pi_n$, where $|P|$ denotes the number of blocks
of~$P$. Let $\one(F)$ be the number of $1$'s in a $01$-filling~$F$. By the palindromity 
of $C_{n+1}(t)$ and the constructions of bijections $\mathcal{C},\mathcal{E}$ and $f$,
we have 
\begin{align*}
C_{n+1}(t)&=\sum_{P\in\NC_{n+1}^{(2)}}t^{|P|}=\sum_{P\in\NC_{n+1}^{(2)}}t^{n+2-|P|}=
\sum_{P\in\NC_{n+1}^{(2)}}t^{\arc_2(P)+1}=\\
&=t\sum_{F\in\mathcal{C}_n^{(2)}}t^{\one(F)}=
t\sum_{i=0}^n{n\choose i}\sum_{F\in\mathcal{E}_i^{(2)}}t^{\one(F)}=\\
&=t\sum_{i=0}^n{n\choose i}\sum_{P\in\NW_i^{(2)}}t^{\arc(P)}=
t\sum_{i=0}^n{n\choose i}\sum_{P\in\NW_i^{(2)}}t^{|P|},
\end{align*}
where the last equality follows from the fact that $\arc(P)=|P|$ for any
$P\in\NW_i^{(2)}$. This shows identity~\eqref{eq:t-CM}, as desired. 

In general, we have $\arc(P)\geq|P|$ for any partition $P$ and $\NC_{n+1}^{(k)}(t)$ is not 
palindromic anymore for $k\geq3$. This is the reason why we could not prove
identity~\eqref{eq:NCNW} for $k\geq3$ by utilizing Theorem~\ref{th:f}. However, we have 
been able to prove~\eqref{eq:NCNW} via an intriguing bijection on $k$-noncrossing 
partitions which we shall introduce next. 
\section{Bijective proof of~\eqref{eq:NCNW}}\label{NCNW}
This section forms the most interesting part of our paper. We will first illustrate our 
bijective proof of~\eqref{eq:NCNW} for noncrossing partitions and then extend it to all
$k$-noncrossing partitions. The extension of our construction from $k=2$ to general $k$ is 
highly nontrivial. We believe it is better to show our framework for the noncrossing 
partition case first, although one can skip Section~\ref{cross2} and go to
Section~\ref{main:k-cross} directly. 

In this section, we let $\Pi_n$ denote the set of partitions of $\{0,1,\ldots,n-1\}$ 
rather than partitions of $[n]$, for convenience's sake. 
\subsection{Noncrossing partitions}\label{cross2}
We give a combinatorial interpretation of identity~\eqref{eq:t-CM}. 

First we interpret the sum $t\sum_{i=0}^n\bi{n}{i}M_i(t)$ as the generating function of
all pairs $(A,\mu)$ such that $A$ is a subset of $\{1,2,\dots,n\}$ and $\mu$ is
a noncrossing (partial) matching whose nodes are elements of $A$ placed on the line
in the natural order. A pair $(A,\mu)$ is weighted by $t^{|\mu|+1}$, where $|\mu|$ is the 
number of blocks of $\mu$. If $A$ is the empty set, then $\mu$ is the empty matching
with weight $t$. 

We now define a combinatorial bijection $\Psi$ from noncrossing partitions in $\Pi_{n+1}$ 
to the set of all pairs $(A,\mu)$ in the above. As it is often the case, an example will do.
Let $n=10$ and
$$
\pi=\{\{0,8,10\},\{1,2,7\},\{3,5,6\},\{4\},\{9\}\}.
$$
This $\pi$ is a noncrossing partition as can be seen below:
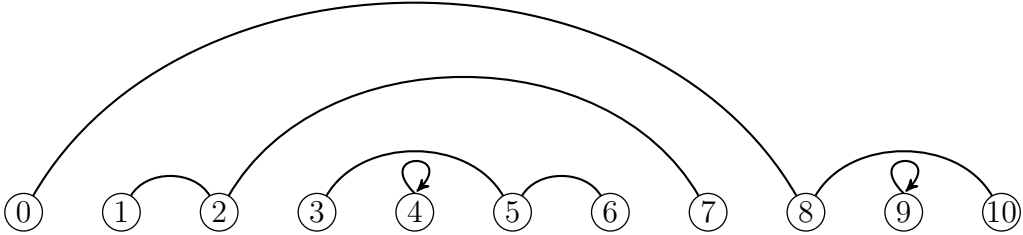
\begin{center}
\begin{tikzpicture}
\SetVertexNormal
\SetGraphUnit{1.3}
\tikzset{VertexStyle/.append style={inner sep=0pt,minimum size=5mm}}
\Vertex{0}
\EA(0){1}\EA(1){2}\EA(2){3}\EA(3){4}\EA(4){5}\EA(5){6}\EA(6){7}
\EA(7){8}\EA(8){9}\EA(9){10}
\tikzset{EdgeStyle/.append style = {bend left = 60}}
\Loop[dist=.8cm, dir=EA](4.north)
\Loop[dist=.8cm, dir=EA](9.north)
\Edge(0)(8)\Edge(8)(10)
\Edge(1)(2)\Edge(2)(7)
\Edge(3)(5)\Edge(5)(6)
\end{tikzpicture}
\end{center}

Consider all blocks in $\pi$ which do not contain $0$:
$\{\{1,2,7\},\{3,5,6\},\{4\},\{9\}\}$.
From each block, delete all integers which are neither the smallest nor
the largest in the block.
Let the resulting set be $\mu$, and let $A$ be the union of all blocks in $\mu$:
$$
(A,\mu)=(\{1,3,4,6,7,9\},\{\{1,7\},\{3,6\},\{4\},\{9\}\})
$$
The next figure shows the elements of $A$, in blue, and the (partial) matching $\mu$.
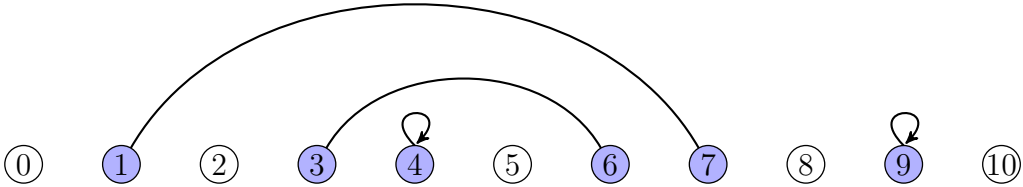
\begin{center}
\begin{tikzpicture}
\SetVertexNormal
\SetGraphUnit{1.3}
\tikzset{VertexStyle/.append style={inner sep=0pt,minimum size=5mm}}
\Vertex{0}
{\tikzset{VertexStyle/.append style={fill=blue!30}}
\EA(0){1}}
\EA(1){2}
{\tikzset{VertexStyle/.append style={fill=blue!30}}
\EA(2){3}\EA(3){4}}
\EA(4){5}
{\tikzset{VertexStyle/.append style={fill=blue!30}}
\EA(5){6}\EA(6){7}}
\EA(7){8}
{\tikzset{VertexStyle/.append style={fill=blue!30}}
\EA(8){9}}
\EA(9){10}
\tikzset{EdgeStyle/.append style = {bend left = 60}}
\Loop[dist=.8cm, dir=EA](4.north)
\Loop[dist=.8cm, dir=EA](9.north)
\Edge(1)(7)
\Edge(3)(6)
\end{tikzpicture}
\end{center}
Let $\Psi(\pi)=(A,\mu)$. Clearly, this is weight-preserving.

The above procedure is reversible. Let $(A,\mu)$ be a pair such that
$A$ is a subset of $\{1,2,\dots,n\}$ and $\mu$ is a noncrossing (partial) matching
whose nodes are elements of $A$ placed on the line in the natural order.

We will construct the corresponding partition $\pi$ of $\{0,1,2,\dots,n\}$ as follows.
Interpret each block $\beta$ in $\mu$ as an interval
$I(\beta)=\{i:\min\{\beta\}\leq i\leq\max\{\beta\}\}$. Let the block of $\pi$ containing $0$ be
$$
\{0,1,2,\dots,n\}\setminus\cup_{\beta\in\mu}I(\beta).
$$
As an example, let $n=10$ and $(A,\mu)=(\{1,3,4,6,7,9\},\{\{1,7\},\{3,6\},\{4\},\{9\}\})$.
The block containing $0$ is $\{0,8,10\}$, shown in red below.
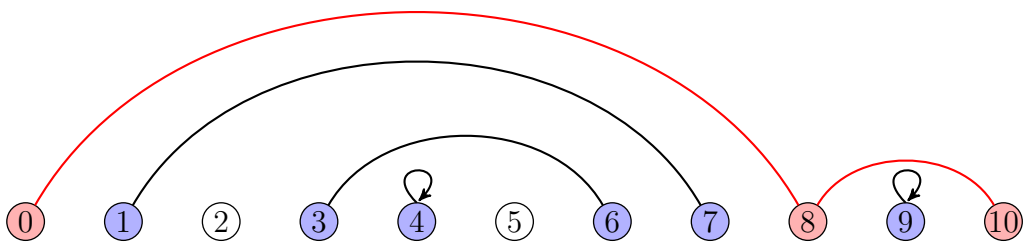
\begin{center}
\begin{tikzpicture}
\SetVertexNormal
\SetGraphUnit{1.3}
\tikzset{VertexStyle/.append style={inner sep=0pt,minimum size=5mm}}

{\tikzset{VertexStyle/.append style={fill=red!30}}\Vertex{0}}
{\tikzset{VertexStyle/.append style={fill=blue!30}}
\EA(0){1}}
\EA(1){2}
{\tikzset{VertexStyle/.append style={fill=blue!30}}
\EA(2){3}\EA(3){4}}
\EA(4){5}
{\tikzset{VertexStyle/.append style={fill=blue!30}}
\EA(5){6}\EA(6){7}}
{\tikzset{VertexStyle/.append style={fill=red!30}}
\EA(7){8}}
{\tikzset{VertexStyle/.append style={fill=blue!30}}
\EA(8){9}}
{\tikzset{VertexStyle/.append style={fill=red!30}}
\EA(9){10}}
\tikzset{EdgeStyle/.append style = {bend left = 60}}
\Loop[dist=.8cm, dir=EA](4.north)
\Loop[dist=.8cm, dir=EA](9.north)
{\tikzset{EdgeStyle/.append style={color=red}}
\Edge(0)(8)\Edge(8)(10)}
\Edge(1)(7)
\Edge(3)(6)
\end{tikzpicture}
\end{center}
Other blocks of $\pi$ are obtained by extending blocks in $\mu$
by the rule:
\begin{quote}
$i\in\{1,2,\dots,n\}\setminus A$ belongs to the block originating
from a block $\beta\in\mu$ if $I(\beta)$ is the smallest interval containing~$i$.
\end{quote}

In our example, two blocks $\{1,7\}$ and $\{3,6\}$ are enlarged, shown in blue below.
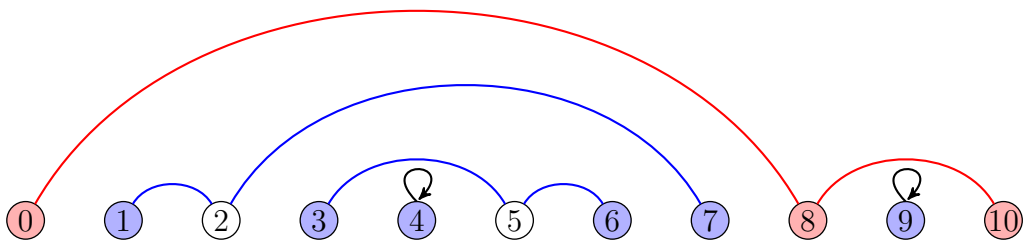
\begin{center}
\begin{tikzpicture}
\SetVertexNormal
\SetGraphUnit{1.3}
\tikzset{VertexStyle/.append style={inner sep=0pt,minimum size=5mm}}
{\tikzset{VertexStyle/.append style={fill=red!30}}\Vertex{0}}
{\tikzset{VertexStyle/.append style={fill=blue!30}}
\EA(0){1}}
\EA(1){2}
{\tikzset{VertexStyle/.append style={fill=blue!30}}
\EA(2){3}\EA(3){4}}
\EA(4){5}
{\tikzset{VertexStyle/.append style={fill=blue!30}}
\EA(5){6}\EA(6){7}}
{\tikzset{VertexStyle/.append style={fill=red!30}}
\EA(7){8}
}
{\tikzset{VertexStyle/.append style={fill=blue!30}}
\EA(8){9}}
{\tikzset{VertexStyle/.append style={fill=red!30}}
\EA(9){10}}
\tikzset{EdgeStyle/.append style = {bend left = 60}}
\Loop[dist=.8cm, dir=EA](4.north)
\Loop[dist=.8cm, dir=EA](9.north)
{\tikzset{EdgeStyle/.append style={color=red}}
\Edge(0)(8)\Edge(8)(10)}
{\tikzset{EdgeStyle/.append style={color=blue}}
\Edge(1)(2)\Edge(2)(7)
\Edge(3)(5)\Edge(5)(6)}
\end{tikzpicture}
\end{center}
So $\Psi^{-1}(\pi)=\{\{0,8,10\},\{1,2,7\},\{3,5,6\},\{4\},\{9\}\}$.

\subsection{$k$-crossing and weak $k$-crossing}
\label{main:k-cross}
Since the block containing $0$ is important in our discussion, we fix the following 
terminology. Recall that the arc diagram of a partition has been defined in Definition~\ref{ArcDiagram}.
\begin{Def}[Red block, Colored arc diagram]
In a partition $P$, the block containing~$0$ is called a {\bf red block}, denoted by
$\redd(P)$, and other blocks are called {\bf black blocks}. The elements in $\redd(P)$ 
are colored red, and other elements are colored black. Arcs in the arc diagram of $P$ between 
red elements are colored red and other arcs are colored black. Such a colored version of
the arc diagram of $P$ is called the {\bf colored arc diagram}, denoted by $D(P)$.

A (weak) $k$-crossing is called a {\bf black (weak) $k$-crossing},
if all its arcs are black; a {\bf red (weak) $k$-crossing}, otherwise.
A partition is called {\bf $k$-crossing} if it has at least one $k$-crossing.
\end{Def}

Recall that $\NC_n^{(k)}$ is the set of all $k$-noncrossing partitions in $\Pi_n$ and 
$\NW_n^{(k)}$ is the set of all enhanced $k$-noncrossing partitions in $\Pi_n$.
Let $\NB_n^{(k)}$ be the set of all partitions $P$ in $\Pi_n$
whose colored arc diagram, $D(P)$, has neither black $k$-crossings nor black weak
$k$-crossings, i.e., has no black enhanced $k$-crossings.

For any subset $A$ of $\{1,2,\dots,n-1\}$, define a subset $\Pi_A$ of $\Pi_n$ by
$$
\Pi_A=\{P\in\Pi_{n}:\redd(P)=\{0,1,\dots,n-1\}\setminus A\}.
$$
Note that $\Pi_n$ is partitioned into $\{\Pi_A\}_{A\subseteq\{1,2,\dots,n-1\}}$,
and there is a natural correspondence between $\Pi_A$ and $\Pi_{|A|}$:
if $A=\{a_1<a_2<\cdots<a_l\}$, then the correspondence is obtained by mapping $a_i$
to $i-1$ for each $i$. 
This correspondence reduces the number of blocks by $1$, since the red block is ignored.
Let us define a subset $\NB_A^{(k)}$ of $\NB_n^{(k)}$ by
$$
\NB_A^{(k)}=\Pi_A\cap \NB_n^{(k)}.
$$
We can see that $\NB_n^{(k)}$ is partitioned into $\{\NB_A^{(k)}\}_{A\subseteq\{1,2,\dots,n-1\}}$,
and there is a natural correspondence between $\NB_A^{(k)}$ and $\NW_{|A|}^{(k)}$, i.e.,
the restriction of the natural correspondence between $\Pi_A$ and $\Pi_{|A|}$.

Define a weight function $w$ on $\Pi_n$ by $w(P)=t^{|P|}$ for each $P\in\Pi_n$,
where $|P|$ denotes the number of blocks in $P$. Since we have
$$
\sum_{P\in \NC_{n+1}^{(k)}}w(P)=\NC_{n+1}^{(k)}(t)
$$
and
\begin{eqnarray*}
\sum_{P\in \NB_{n+1}^{(k)}}w(P)
&=&\sum_{A\subseteq\{1,\dots,n\}}\sum_{P\in \NB_A^{(k)}}w(P)\\
&=&\sum_{A\subseteq\{1,\dots,n\}}t\sum_{P\in \NW_{|A|}^{(k)}}w(P)\\
&=&t\sum_{i=0}^n\bi{n}{i}\NW_{i}^{(k)}(t),
\end{eqnarray*}
identity~\eqref{eq:NCNW} is equivalent to the following theorem.
\begin{theorem}\label{thm:NCCNW}
For all $n$ and $k$, there exists a weight-preserving combinatorial bijection 
$
\Phi:\NB_{n+1}^{(k)}\to \NC_{n+1}^{(k)}
$ 
proving
\begin{equation*}\label{eq:NCCNW}
\sum_{P\in \NB_{n+1}^{(k)}}w(P)=\sum_{P\in \NC_{n+1}^{(k)}}w(P).
\end{equation*}
\end{theorem}
\subsection{The combinatorial bijection $\Phi$}
Recall that a weak $k$-crossing of $P\in\Pi_n$ is a $k$-subset
$(i_1,j_1),(i_2,j_2),\ldots,(i_k,j_k)$ of arcs such that 
$$
i_1<i_2<\cdots<\blue{i_k= j_1}<j_2<\cdots<j_k.
$$
The position $c$, $c=i_k=j_1$, is called the {\em center} of the weak $k$-crossing.
We will say that a node~$a$ is {\bf under a $k$-crossing}
$(i_1,j_1),(i_2,j_2),\ldots,(i_k,j_k)$ if $i_k<a<j_1$, i.e.,
$$
i_1<i_2<\cdots<\blue{i_k<a<j_1}<j_2<\cdots<j_k.
$$

Since a $(k-1)$-noncrossing partition has no enhanced $k$-crossings and
an enhanced $k$-noncrossing partition has no $k$-crossings,
$$
\NC_n^{(k-1)}\subseteq \NW_n^{(k)}\subseteq \NC_n^{(k)}
$$
for all $k\geq3$. The combinatorial bijection 
$$
\Phi:\NB_{n+1}^{(k)}\to \NC_{n+1}^{(k)}
$$
is constructed by the following steps:
\begin{enumerate}
\item Let $P=\{B_0,B_1,\dots,B_l\}\in \NB_{n+1}^{(k)}$ with $\redd(P)=B_0$.
If $P\in\NB_{n+1}^{(k-1)}$ then $P$ belongs to $\NC_{n+1}^{(k)}$ and we can
set $\Phi(P)=P$. Otherwise, 
$P\in\NB_{n+1}^{(k)}\setminus\NB_{n+1}^{(k-1)}$.
\item Start with $D(P)$, the colored arc diagram of $P$.
\item If there exists a red node under a black $(k-1)$-crossing in $D(P)$,
do {\bf `enhanced left shift'} on $D(P)$, i.e.,
\begin{itemize}
\item let $a$ be the {\bf smallest} such red node,
\item let $(i_1,j_1),(i_2,j_2),\ldots,(i_{k-1},j_{k-1})$ be the {\bf innermost}
(that is to say the word $(j_1,j_2,\ldots, j_{k-1})$ is smallest in the lexicographic order)
black $(k-1)$-crossing covering~$a$,
\item change arcs forming a black $(k-1)$-crossing
$(i_1,j_1),(i_2,j_2),\ldots,(i_{k-1},j_{k-1})$
into arcs of a black weak $k$-crossing
$(i_1,a),(i_2,j_1),\ldots,(i_{k-1},j_{k-2}),(a,j_{k-1})$
with $a$ as the center,
\item set $B_0=B_0\setminus\{a\}$, and let $\tilde{P}$ denote the resulting partition.
\end{itemize}
Repeat this step until the colored arc diagram of $\tilde{P}$ has no red node under
a black $(k-1)$-crossing. The resulting partition $\tilde{P}$ has no black $k$-crossing,
which is ensured by Lemma~\ref{shift}. 

\item If $D(\tilde{P})$ has no red $k$-crossing, then set $\Phi(P)=\tilde{P}$; otherwise,
do {\bf`cyclic rotation'} on $D(\tilde{P})$, i.e.,
\begin{itemize}
\item find the {\bf rightmost} red arc in a $k$-crossing, say \red{$(i,j)$},
\item let $(i_1,j_1),(i_2,j_2),\ldots,(i_{k},j_{k})$ be the {\bf greatest}, in the
lexicographic order of $(j_1,j_2,\dots,j_k)$,
$k$-crossing with \red{$(i_p,j_p)=(i,j)$} (here $p$ is always greater than $1$ in the process).
\item change arcs forming a $k$-crossing
$(i_1,j_1),(i_2,j_2),\ldots,\red{(i_p,j_p)},\ldots,(i_{k},j_{k})$ into arcs
$$
(i_1,j_2),(i_2,j_3),\ldots,(i_{p-1},j_p),\red{(i_p,j_1)},(i_{p+1},j_{p+1}),\ldots,
(i_k,j_k),
$$
where $(i_p,j_1)$ is recolored red,
\item color $j_p$ black, $j_1$ red, and recolor the nodes in the blocks containing
$j_p$ and $j_1$ accordingly. 
\end{itemize}
Repeat this process until the resulting colored arc diagram has no red $k$-crossing.
The partition $P'$ corresponding to the resulting colored arc diagram has no black
$k$-crossing, which is proved in Lemma~\ref{cyc_rotation}.
\item Finally we end up with a partition $P'$ in $\NC_{n+1}^{(k)}$. Set $\Phi(P)=P'$.
\end{enumerate}
Note that if $P\in \NB_{n+1}^{(k)}$ has a $(k-1)$-crossing then so does $\Phi(P)$ but
the converse is not true. The reader is invited to check that $\Phi$ agrees with $\Psi^{-1}$
when $k=2$, even though they are defined differently.
\begin{example}
An example of $\Phi$ with $(n,k)=(16,3)$ and $P\in\NB_{17}^{(3)}$:
$$
P=\{\{0,4,8,15\},\{1,3,10\},\{2,11\},\{5,16\},\{6,13\},\{7,9,12,14\}\}.
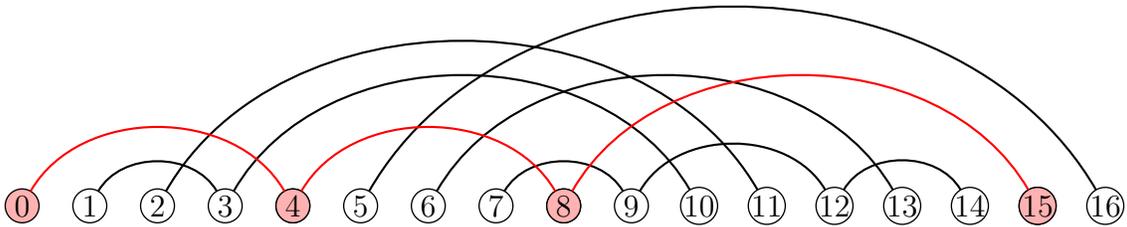
$$
\begin{center}
\begin{tikzpicture}
\SetVertexNormal
\SetGraphUnit{\graphsize}
\tikzset{VertexStyle/.append style={inner sep=0pt,minimum size=4.5mm}}
{\tikzset{VertexStyle/.append style={fill=red!30}}\Vertex{0}}
\foreach \1/\2 in {0/1,1/2,2/3,3/4,4/5,5/6,6/7,7/8,8/9,9/10,10/11,11/12,12/13,13/14,14/15,15/16}{\EA(\1){\2}}
\tikzset{VertexStyle/.append style={fill=red!30}}
\foreach \1/\2 in {3/4,7/8,14/15}{\EA(\1){\2}}
\tikzset{VertexStyle/.append style={fill=blue!30}}
\foreach \1/\2 in {}{\EA(\1){\2}}
\tikzset{VertexStyle/.append style={fill=green!30}}
\foreach \1/\2 in {}{\EA(\1){\2}}
\tikzset{EdgeStyle/.append style={bend left=60}}
\foreach \1/\2 in {1/3,2/11,5/16,3/10,6/13,9/12,12/14,7/9}{\Edge(\1)(\2)}
\tikzset{EdgeStyle/.append style={color=red}}
\foreach \1/\2 in {0/4,4/8,8/15}{\Edge(\1)(\2)}
\tikzset{EdgeStyle/.append style={color=blue}}
\foreach \1/\2 in {}{\Edge(\1)(\2)}
\tikzset{EdgeStyle/.append style={color=green}}
\foreach \1/\2 in {}{\Edge(\1)(\2)}
\end{tikzpicture}
\end{center}
Red $8$ is under four black $2$-crossings of which the innermost is
$(3,10),(6,13)$. Make $8$ the center of a black weak $3$-crossing,
$(3,8),(6,10),(8,13)$, and uncolor~$8$.
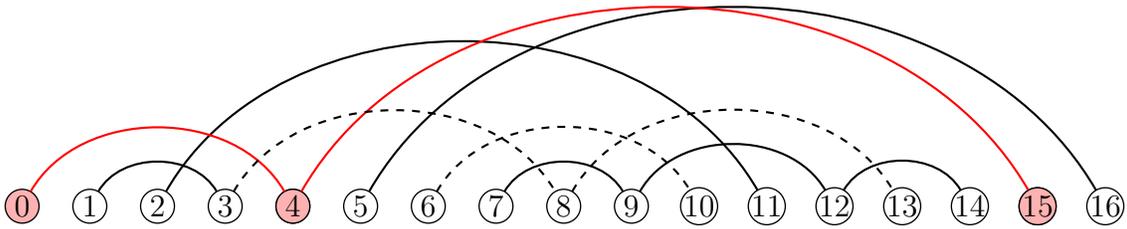
\vspace{-2mm}
\begin{center}
\begin{tikzpicture}
\SetVertexNormal
\SetGraphUnit{\graphsize}
\tikzset{VertexStyle/.append style={inner sep=0pt,minimum size=4.5mm}}
{\tikzset{VertexStyle/.append style={fill=red!30}}\Vertex{0}}
\foreach \1/\2 in {0/1,1/2,2/3,3/4,4/5,5/6,6/7,7/8,8/9,9/10,10/11,11/12,12/13,13/14,14/15,15/16}{\EA(\1){\2}}
\tikzset{VertexStyle/.append style={fill=red!30}}
\foreach \1/\2 in {3/4,14/15}{\EA(\1){\2}}
\tikzset{VertexStyle/.append style={fill=blue!30}}
\foreach \1/\2 in {}{\EA(\1){\2}}
\tikzset{VertexStyle/.append style={fill=green!30}}
\foreach \1/\2 in {}{\EA(\1){\2}}
\tikzset{EdgeStyle/.append style={bend left=60}}
\foreach \1/\2 in {1/3,2/11,5/16,9/12,12/14,7/9}{\Edge(\1)(\2)}
\tikzset{EdgeStyle/.append style={color=red}}
\foreach \1/\2 in {0/4,4/15}{\Edge(\1)(\2)}
{\tikzset{EdgeStyle/.append style={color=black,dashed}}
\foreach \1/\2 in {3/8,8/13,6/10}{\Edge(\1)(\2)}}
\tikzset{EdgeStyle/.append style={color=green}}
\foreach \1/\2 in {}{\Edge(\1)(\2)}
\end{tikzpicture}
\end{center}
Dashed arcs form the weak $3$-crossing.
Arcs $(2,11),\red{(4,15)},(5,16)$ form a red $3$-crossing. Do `cyclic rotation':
$(2,11),\red{(4,15)},(5,16)\to (2,15),\red{(4,11)},(5,16)$.
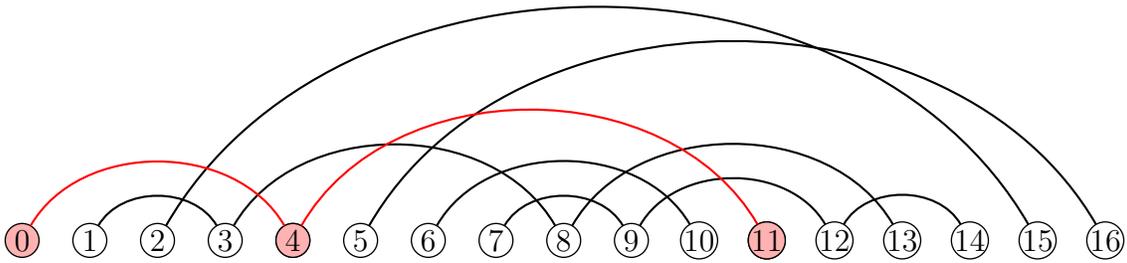
\vspace{-2mm}
\begin{center}
\begin{tikzpicture}
\SetVertexNormal
\SetGraphUnit{\graphsize}
\tikzset{VertexStyle/.append style={inner sep=0pt,minimum size=4.5mm}}
{\tikzset{VertexStyle/.append style={fill=red!30}}\Vertex{0}}
\foreach \1/\2 in {0/1,1/2,2/3,3/4,4/5,5/6,6/7,7/8,8/9,9/10,10/11,11/12,12/13,13/14,14/15,15/16}{\EA(\1){\2}}
\tikzset{VertexStyle/.append style={fill=red!30}}
\foreach \1/\2 in {3/4,10/11}{\EA(\1){\2}}
\tikzset{VertexStyle/.append style={fill=blue!30}}
\foreach \1/\2 in {}{\EA(\1){\2}}
\tikzset{VertexStyle/.append style={fill=green!30}}
\foreach \1/\2 in {}{\EA(\1){\2}}
\tikzset{EdgeStyle/.append style={bend left=60}}
\foreach \1/\2 in {1/3,2/15,5/16,9/12,12/14,7/9,3/8,8/13,6/10}{\Edge(\1)(\2)}
\tikzset{EdgeStyle/.append style={color=red}}
\foreach \1/\2 in {0/4,4/11}{\Edge(\1)(\2)}
\tikzset{EdgeStyle/.append style={color=blue}}
\foreach \1/\2 in {}{\Edge(\1)(\2)}
\tikzset{EdgeStyle/.append style={color=green}}
\foreach \1/\2 in {}{\Edge(\1)(\2)}
\end{tikzpicture}
\end{center}
Arcs $(3,8),\red{(4,11)},(5,16)$ form a red $3$-crossing. Do `cyclic rotation':
$(3,8),\red{(4,11)},(5,16)\to (3,11),\red{(4,8)},(5,16)$.
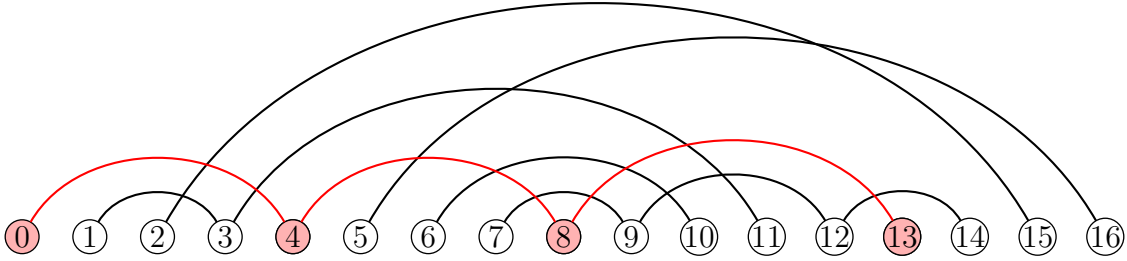
\vspace{-2mm}
\begin{center}
\begin{tikzpicture}
\SetVertexNormal
\SetGraphUnit{\graphsize}
\tikzset{VertexStyle/.append style={inner sep=0pt,minimum size=4.5mm}}
{\tikzset{VertexStyle/.append style={fill=red!30}}\Vertex{0}}
\foreach \1/\2 in {0/1,1/2,2/3,3/4,4/5,5/6,6/7,7/8,8/9,9/10,10/11,11/12,12/13,13/14,14/15,15/16}{\EA(\1){\2}}
\tikzset{VertexStyle/.append style={fill=red!30}}
\foreach \1/\2 in {3/4,7/8,12/13}{\EA(\1){\2}}
\tikzset{VertexStyle/.append style={fill=blue!30}}
\foreach \1/\2 in {}{\EA(\1){\2}}
\tikzset{VertexStyle/.append style={fill=green!30}}
\foreach \1/\2 in {}{\EA(\1){\2}}
\tikzset{EdgeStyle/.append style={bend left=60}}
\foreach \1/\2 in {1/3,2/15,3/11,5/16,9/12,12/14,7/9,6/10}{\Edge(\1)(\2)}
\tikzset{EdgeStyle/.append style={color=red}}
\foreach \1/\2 in {0/4,4/8,8/13}{\Edge(\1)(\2)}
\tikzset{EdgeStyle/.append style={color=blue}}
\foreach \1/\2 in {}{\Edge(\1)(\2)}
\tikzset{EdgeStyle/.append style={color=green}}
\foreach \1/\2 in {}{\Edge(\1)(\2)}
\end{tikzpicture}
\end{center}
The last colored arc diagram corresponds to $\Phi(P)\in\NC_{17}^{(3)}$:
$$
\Phi(P)=(\{0,4,8,13\},\{1,3,11\},\{2,15\},\{5,16\},\{6,10\},\{7,9,12,14\}).
$$
\end{example}

The crucial reason why $\Phi$ is reversible is that any `cyclic rotation' to a red $k$-crossing
leaves a trace, i.e., a red node under a black $(k-1)$-crossing.
In the following, we show that $\Phi$ is a bijection by defining its inverse explicitly:
$$
\Phi^{-1}:\NC_{n+1}^{(k)}\to \NB_{n+1}^{(k)}. 
$$
\begin{enumerate}
\item Let $P=\{B_0,B_1,\dots,B_l\}\in \NC_{n+1}^{(k)}$ with $\redd(P)=B_0$.
If $P\in\NB_{n+1}^{(k-1)}$ then set $\Phi^{-1}(P)=P$.
Otherwise, $P\in\NC_{n+1}^{(k)}\setminus\NB_{n+1}^{(k-1)}$.
\item Start with $D(P)$, the colored arc diagram of $P$.
\item If there is a red node under a black $(k-1)$-crossing in $D(P)$, then undo {\bf `cyclic rotation'},
i.e.,
\begin{itemize}
\item let $a$ be the {\bf smallest} such node and $a'$ the largest node among nodes in $\redd(P)$ 
which are smaller than $a$,
\item let $(i_1,j_1),(i_2,j_2),\ldots,(i_{k-1},j_{k-1})$ be the innermost black $(k-1)$-crossing with
$$
i_1<i_2<\cdots<i_{k-1}<a<j_1<j_2<\cdots<j_{k-1},
$$
and find $t$ such that $1\leq t<k-1$ and $i_t<a'<i_{t+1}$,
\item change arcs
$$
(i_1,j_1),(i_2,j_2),\ldots,(i_t,j_t),\red{(a',a)},(i_{t+1},j_{t+1}),\ldots,(i_{k-1},j_{k-1})
$$
into arcs
$$
(i_1,a),(i_2,j_1),\ldots,(i_t,j_{t-1}),\red{(a',j_t)},(i_{t+1},j_{t+1}),\ldots,(i_{k-1},j_{k-1}),
$$
and adjust the colors of nodes and arcs, of the resulting diagram so that the block 
containing $0$ is red and other blocks are black, and let the corresponding partition be~$P'$. 
\end{itemize}
Repeat this step with $D(P')$ until the colored arc diagram of the resulting partition
has no red node under black $(k-1)$-crossings.
\item If $D(P')$ has no black weak $k$-crossing, then set $\Phi^{-1}(P)=P'$;
otherwise,
\begin{itemize}
\item let $a$ be the {\bf largest} node among the centers of black weak $k$-crossings
of $D(P')$,
\item change arcs of the {\bf outermost} black weak $k$-crossing with center $a$
$$
(i_1,a),(i_2,j_1),\ldots,(i_{k-1},j_{k-2}),(a,j_{k-1})
$$
into arcs of a $(k-1)$-crossing
$$
(i_1,j_1),(i_2,j_2),\ldots,(i_{k-1},j_{k-1}),
$$
set $B_0=B_0\cup\{a\}$, and let $\tilde{P}$ denote the resulting partition.
\end{itemize}
Repeat this step until the resulting partition has no black weak $k$-crossings.
\item Finally we end up with a partition $\tilde{P}\in\NB_{n+1}^{(k)}$. Set $\Phi^{-1}(P)=\tilde{P}$.
\end{enumerate}
\begin{example}
An example of $\Phi^{-1}$ with $(n,k)=(16,4)$ and $P\in\NC_{17}^{(4)}$:
$$
P=\{\{0,3,5,10,13\},\{1,6,8,12\},\{2,9,15\},\{4,11,16\},\{7,14\}\}.
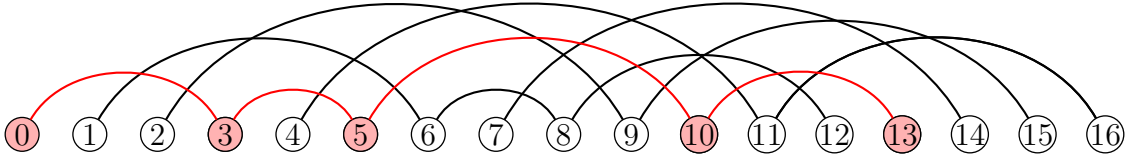
$$
\begin{center}
\begin{tikzpicture}
\SetVertexNormal
\SetGraphUnit{\graphsize}
\tikzset{VertexStyle/.append style={inner sep=0pt,minimum size=4.5mm}}
{\tikzset{VertexStyle/.append style={fill=red!30}}\Vertex{0}}
\foreach \1/\2 in {0/1,1/2,2/3,3/4,4/5,5/6,6/7,7/8,8/9,9/10,10/11,11/12,12/13,13/14,14/15,15/16}{\EA(\1){\2}}
\tikzset{VertexStyle/.append style={fill=red!30}}
\foreach \1/\2 in {2/3,4/5,9/10,12/13}{\EA(\1){\2}}
\tikzset{VertexStyle/.append style={fill=blue!30}}
\foreach \1/\2 in {}{\EA(\1){\2}}
\tikzset{VertexStyle/.append style={fill=green!30}}
\foreach \1/\2 in {}{\EA(\1){\2}}
\tikzset{EdgeStyle/.append style={bend left=60}}
\foreach \1/\2 in {1/6,2/9,4/11,7/14,11/16,6/8,8/12,9/15,11/16}{\Edge(\1)(\2)}
\tikzset{EdgeStyle/.append style={color=red}}
\foreach \1/\2 in {0/3,3/5,5/10,10/13}{\Edge(\1)(\2)}
\tikzset{EdgeStyle/.append style={color=blue}}
\foreach \1/\2 in {}{\Edge(\1)(\2)}
\tikzset{EdgeStyle/.append style={color=green}}
\foreach \1/\2 in {}{\Edge(\1)(\2)}
\end{tikzpicture}
\end{center}
Red $5$ is under a black $3$-crossing: $(1,6),(2,9),(4,11)$.
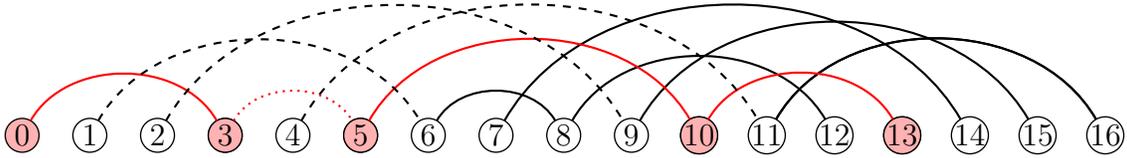
\begin{center}
\begin{tikzpicture}
\SetVertexNormal
\SetGraphUnit{\graphsize}
\tikzset{VertexStyle/.append style={inner sep=0pt,minimum size=4.5mm}}
{\tikzset{VertexStyle/.append style={fill=red!30}}\Vertex{0}}
\foreach \1/\2 in {0/1,1/2,2/3,3/4,4/5,5/6,6/7,7/8,8/9,9/10,10/11,11/12,12/13,13/14,14/15,15/16}{\EA(\1){\2}}
\tikzset{VertexStyle/.append style={fill=red!30}}
\foreach \1/\2 in {2/3,4/5,9/10,12/13}{\EA(\1){\2}}
\tikzset{VertexStyle/.append style={fill=blue!30}}
\foreach \1/\2 in {}{\EA(\1){\2}}
\tikzset{VertexStyle/.append style={fill=green!30}}
\foreach \1/\2 in {}{\EA(\1){\2}}
\tikzset{EdgeStyle/.append style={bend left=60}}
\foreach \1/\2 in {7/14,11/16,6/8,8/12,9/15,11/16}{\Edge(\1)(\2)}
\tikzset{EdgeStyle/.append style={color=red}}
\foreach \1/\2 in {0/3,5/10,10/13}{\Edge(\1)(\2)}
{\tikzset{EdgeStyle/.append style={color=black,dashed}}
\foreach \1/\2 in {1/6,2/9,4/11}{\Edge(\1)(\2)}}
{\tikzset{EdgeStyle/.append style={color=red,dotted}}
\foreach \1/\2 in {3/5}{\Edge(\1)(\2)}}
\end{tikzpicture}
\end{center}
Undo `cyclic rotation': $(1,6),(2,9),\red{(3,5)},(4,11)\to (1,5),(2,6),\red{(3,9)},(4,11)$.
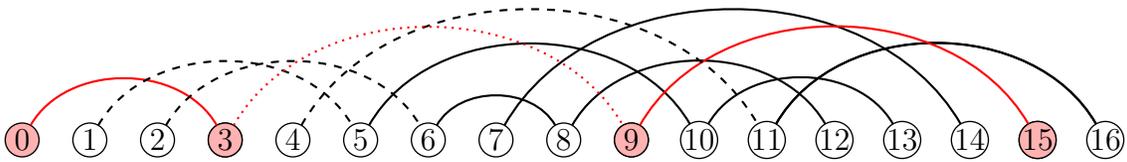
\begin{center}
\begin{tikzpicture}
\SetVertexNormal
\SetGraphUnit{\graphsize}
\tikzset{VertexStyle/.append style={inner sep=0pt,minimum size=4.5mm}}
{\tikzset{VertexStyle/.append style={fill=red!30}}\Vertex{0}}
\foreach \1/\2 in {0/1,1/2,2/3,3/4,4/5,5/6,6/7,7/8,8/9,9/10,10/11,11/12,12/13,13/14,14/15,15/16}{\EA(\1){\2}}
\tikzset{VertexStyle/.append style={fill=red!30}}
\foreach \1/\2 in {2/3,8/9,14/15}{\EA(\1){\2}}
\tikzset{VertexStyle/.append style={fill=blue!30}}
\foreach \1/\2 in {}{\EA(\1){\2}}
\tikzset{VertexStyle/.append style={fill=green!30}}
\foreach \1/\2 in {}{\EA(\1){\2}}
\tikzset{EdgeStyle/.append style={bend left=60}}
\foreach \1/\2 in {5/10,7/14,11/16,6/8,8/12,11/16,10/13}{\Edge(\1)(\2)}
\tikzset{EdgeStyle/.append style={color=red}}
\foreach \1/\2 in {0/3,9/15}{\Edge(\1)(\2)}
{\tikzset{EdgeStyle/.append style={color=black,dashed}}
\foreach \1/\2 in {1/5,2/6,4/11}{\Edge(\1)(\2)}}
{\tikzset{EdgeStyle/.append style={color=red,dotted}}
\foreach \1/\2 in {3/9}{\Edge(\1)(\2)}}
\end{tikzpicture}
\end{center}
Arcs $(4,11),(8,12),(10,13),(11,16)$ form a black weak $4$-crossing. Change them to
form a black $3$-crossing $(4,12),(8,13),(10,16)$ and color $11$ red.
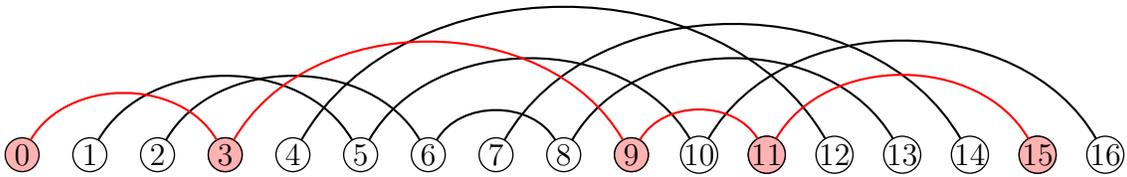
\begin{center}
\begin{tikzpicture}
\SetVertexNormal
\SetGraphUnit{\graphsize}
\tikzset{VertexStyle/.append style={inner sep=0pt,minimum size=4.5mm}}
{\tikzset{VertexStyle/.append style={fill=red!30}}\Vertex{0}}
\foreach \1/\2 in {0/1,1/2,2/3,3/4,4/5,5/6,6/7,7/8,8/9,9/10,10/11,11/12,12/13,13/14,14/15,15/16}{\EA(\1){\2}}
\tikzset{VertexStyle/.append style={fill=red!30}}
\foreach \1/\2 in {2/3,8/9,10/11,14/15}{\EA(\1){\2}}
\tikzset{VertexStyle/.append style={fill=blue!30}}
\foreach \1/\2 in {}{\EA(\1){\2}}
\tikzset{VertexStyle/.append style={fill=green!30}}
\foreach \1/\2 in {}{\EA(\1){\2}}
\tikzset{EdgeStyle/.append style={bend left=60}}
\foreach \1/\2 in {1/5,2/6,4/12,5/10,7/14,6/8,8/13,10/16}{\Edge(\1)(\2)}
\tikzset{EdgeStyle/.append style={color=red}}
\foreach \1/\2 in {0/3,3/9,9/11,11/15}{\Edge(\1)(\2)}
{\tikzset{EdgeStyle/.append style={color=black,dashed}}
\foreach \1/\2 in {}{\Edge(\1)(\2)}}
{\tikzset{EdgeStyle/.append style={color=red,dotted}}
\foreach \1/\2 in {}{\Edge(\1)(\2)}}
\end{tikzpicture}
\end{center}
The last colored arc diagram corresponds to $\Phi^{-1}(P)\in\NB_{17}^{(4)}$:
$$
\Phi^{-1}(P)=\{\{0,3,9,11,15\},\{1,5,10,16\},\{2,6,8,13\},\{4,12\},\{7,14\}\}.
$$
\end{example}
It can be checked routinely that $\Phi^{-1}(\Phi(P))=P$ for each $P\in \NB_{n+1}^{(k)}$, i.e., $\Phi$ is
injective. Therefore, $\Phi$ is bijective due to the cardinality reason in view of Theorem~\ref{th:f}.
Moreover, it is clear that each step of $\Phi$ preserves the number of blocks, which completes
the proof of Theorem~\ref{thm:NCCNW}. In the rest of this section, we will prove two technical
lemmas which insure that the algorithm $\Phi$ is well-defined. 
\subsection{Two technical lemmas}
\label{sec:techni}
For two arcs $\alpha=(i,j)$ and $\alpha'=(i',j')$ of a partition, it is convenient to use
the notation $\alpha\prec\alpha'$ whenever $i<i'$ and $j<j'$. Some parts of the discussion
below become more intuitive if one uses the representation~$\mathcal{C}$ (introduced in
Section~\ref{trian}), as $01$-fillings of triangular shapes, for set partitions.
Under the representation~$\mathcal{C}$, an arc is identified with a $1$, while a $k$-crossing
is identified with a proper SE-chain of length~$k$. 
\begin{lemma}\label{shift}
The `enhanced left shift' operation defined in the algorithm $\Phi$ does not create any black
$k$-crossing for a partition without any black $k$-crossing.
\end{lemma}

\begin{proof}
Suppose that $a$ is a red node under an innermost black $(k-1)$-crossing 
$$
\alpha_1=(i_1,j_1),\alpha_2=(i_2,j_2),\ldots,\alpha_{k-1}=(i_{k-1},j_{k-1})
$$
in $D(P)$, the colored arc diagram of a partition $P$ without any black $k$-crossing. Let 
$$
\beta_1=(i_1,a),\beta_2=(i_2,j_1),\ldots,\beta_{k-1}=(i_{k-1},j_{k-2}),\beta_k=(a,j_{k-1})
$$
be the corresponding black weak $k$-crossing with $a$ as the center in the partition $\tilde{P}$
after applying the `enhanced left shift'. 

Now we proceed to show that $\tilde{P}$ has no black $k$-crossing. If not, suppose that $\tilde{P}$
contains a black $k$-crossing $\gamma=(\gamma_1,\gamma_2,\ldots,\gamma_k)$, where
$\gamma_i=(r_i,s_i)$. The following fact is crucial in our approach. 
 
\begin{Fact}\label{fact:1}
For $u<u'$, if $\{\gamma_i:\beta_{u}\prec \gamma_i\prec \beta_{u'}\}\cap\{\beta_i: u<i<u'\}=\emptyset$, then
$$
|\{\gamma_i:\beta_{u}\prec \gamma_i\prec \beta_{u'}\}|\leq u'-u-1. 
$$
\end{Fact}
\begin{proof}
There is nothing to prove if $\{\gamma_i:\beta_{u}\prec \gamma_i\prec \beta_{u'}\}=\emptyset$.
So we can suppose that $\{\gamma_i:\beta_{u}\prec \gamma_i\prec \beta_{u'}\}=\{\gamma_i: v\leq i\leq v'\}$
for some $v\leq v'$. We claim that $v'-v\leq u'-u-2$, which is equivalent to our statement.
Otherwise, $v'-v\geq u'-u-1$ and if $s_{v}<j_u$, then in $D(P)$ the red node $a$ is under
the black $(k-1)$-crossing 
$$
\alpha_1,\alpha_2,\ldots,\alpha_{u-1},\gamma_{v},\gamma_{v+1},\ldots,\gamma_{v+u'-u-1},
\alpha_{u'},\alpha_{u'+1},\ldots,\alpha_{k-1},
$$
which is more internal than $\alpha_1,\alpha_2,\ldots,\alpha_{k-1}$, a contradiction.
If $j_u<s_{v}$, then $D(P)$ contains the black $k$-crossing 
$$
\alpha_1,\alpha_2,\ldots,\alpha_{u},\gamma_{v},\gamma_{v+1},\ldots,\gamma_{v+u'-u-1},
\alpha_{u'},\alpha_{u'+1},\ldots,\alpha_{k-1},
$$
a contradiction again. This finishes the proof.
\end{proof}
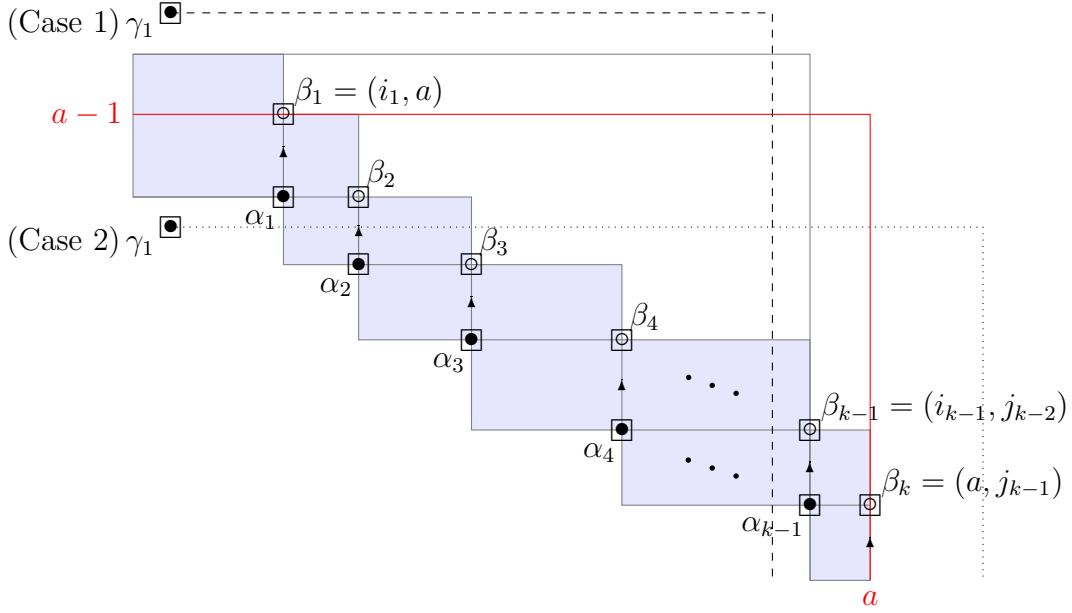
\begin{figure}
\begin{center}
\begin{tikzpicture}[scale=1]
\draw[color=gray,fill=blue!10!white] (0,6) rectangle (2,4.1);
\draw[color=gray,fill=blue!10!white] (2,3.2) rectangle (3,5.2);
\draw[color=gray,fill=blue!10!white] (3,2.2) rectangle (4.5,4.1);
\draw[color=gray,fill=blue!10!white] (4.5,1) rectangle (6.5,3.2);
\draw[color=gray,fill=blue!10!white] (6.5,0) rectangle (9,2.2);
\draw[color=gray,fill=blue!10!white] (9,-1) rectangle (9.8,1);
\draw[color=gray](0,6)--(9,6)--(9,-1)(0,4.1)--(3,4.1)(3,3.2)--(4.5,3.2)(4.5,2.2)--(6.5,2.2)(6.5,1)--(9,1)(9,0)--(9.8,0) ;
\draw[dashed]
(0.5,6.55)--(8.5,6.55)--(8.5,-1);
\node at (0.5,6.55) {$\square$}; \node at (0.5,6.55){$\bullet$};\node at (-0.7, 6.4) {{(Case 1)\,$\gamma_{1}$}};
\draw[dotted]
(0.5,3.7)--(11.3,3.7)--(11.3,-1);
\node at (0.5,3.7) {$\square$}; \node at (0.5,3.7){$\bullet$};\node at (-0.7, 3.5) {{(Case 2)\,$\gamma_{1}$}};
\draw[color=red](0,5.2)--(9.8,5.2)--(9.8,-1);
\node at (2,5.2) {$\square$};\node at (2, 5.2) {$\circ$};\node[right] at (2, 5.5) {$\beta_1=(i_1,a)$};
\node at (2,4.6) {\vector(0,1){0.1}};
\node at (2,4.1) {$\square$}; \node at (2,4.1) {$\bullet$}; \node at (1.7, 3.8) {$\alpha_1$};
\node at (3,4.1) {$\square$};\node at (3,4.1) {$\circ$};\node at (3.3,4.4) {$\beta_2$};
\node at (3,3.55) {\vector(0,1){0.1}};
\node at (3,3.2) {$\square$}; \node at (3,3.2){$\bullet$}; \node at (2.7, 2.9) {$\alpha_2$};
\node at (4.5,3.2) {$\square$}; \node at (4.5,3.2){$\circ$}; \node at (4.8, 3.5) {$\beta_3$};
\node at (4.5,2.6) {\vector(0,1){0.1}};
\node at (4.5,2.2) {$\square$}; \node at (4.5,2.2){$\bullet$};\node at (4.2, 1.9) {$\alpha_3$};
\node at (6.5,2.2) {$\square$}; \node at (6.5,2.2){$\circ$};\node at (6.8, 2.5) {$\beta_4$};
\node at (6.5,1.5) {\vector(0,1){0.1}};
\node at (6.5,1) {$\square$}; \node at (6.5,1){$\bullet$};\node at (6.2, 0.7) {$\alpha_4$};
\node at (7.7,1.7) {\Huge $\ddots$};
\node at (7.7,0.6) {\Huge $\ddots$};
\node at (9,1) {$\square$}; \node at (9,1){$\circ$};\node[right] at (9, 1.3) {$\beta_{k-1}=(i_{k-1},j_{k-2})$};
\node at (9,0.4) {\vector(0,1){0.1}};
\node at (9,0) {$\square$}; \node at (9,0){$\bullet$};\node at (8.5, -0.3) {$\alpha_{k-1}$};
\node at (9.8,0) {$\square$}; \node at (9.8,0){$\circ$};\node[right] at (9.8,0.3) {$\beta_{k}=(a,j_{k-1})$};
\node at (9.8,-0.6) {\vector(0,1){0.1}};
\node[below] at (9.8,-1) {\red{$a$}};
\node[left] at (0,5.2) {\red{$a-1$}};
\end{tikzpicture}
\caption{The `enhanced left shift' on $01$-fillings of triangular shape. There is no $\gamma$ arc inside the shadow.
The position of the cell of arc $(i_1,a)$ is $(a-1,i_1)$.}
\label{fig:leftshift}
\end{center}
\end{figure}
It follows from Fact~\ref{fact:1} that one can choose the black $k$-crossing $\gamma$ of $\tilde{P}$
properly so that the intersection 
$$
\{\beta_i: 1\leq i\leq k\}\cap\{\gamma_i: 1\leq i\leq k\}=\{\beta_{u}=\gamma_{v}\prec\beta_{u+1}=
\gamma_{v+1}\prec\cdots\prec\beta_{u'}=\gamma_{v'}\}
$$
for some $1\leq u\leq u'$ and $u'-u=v'-v$. We need to distinguish two cases according to
whether $s_1$ is greater than $a$ or not.
In the following proof, we will use the $01$-filling representation~$\mathcal{C}$ (see Fig.~\ref{fig:leftshift}).
\begin{itemize}
\item {Case 1: $s_1<a$.}\\
In this case, there is the smallest $w$, $u'\leq w\leq k-1$, such that in
$\mathcal{C}(P)$ no arc from $\{\gamma_i: v'<i\leq k\}$ appears to the northeast of $\beta_{w+1}$.
By Fact~\ref{fact:1},
$$
|\{\gamma_i:\beta_{u'}\prec \gamma_i\prec \beta_{w+1}\}|\leq w-u'. 
$$
Thus, we have $\alpha_{w}\prec\gamma_{w+v-u+1}=\gamma_{w+v'-u'+1}$ and so $\mathcal{C}(P)$
contains the proper (black) SE-chain 
$$
\gamma_1,\gamma_2,\ldots,\gamma_{v-1},\alpha_u,\alpha_{u+1},\ldots,\alpha_{w},
\gamma_{w+v-u+1},\gamma_{w+v-u+2},\ldots,\gamma_{k}
$$
of length $k$, a contradiction.
\item {Case 2: $s_1>a$.}\\
In this case, there is the greatest $w$, $1\leq w\leq u-1$, such that in
$\mathcal{C}(P)$ no arc from $\{\gamma_i: 1\leq i< v\}$ appears to the northeast of $\beta_{w}$.
Then by Fact~\ref{fact:1}, we have 
$$
|\{\gamma_i:\beta_{w}\prec \gamma_i\prec \beta_{u}\}|\leq u-w-1. 
$$
Therefore, $\gamma_{v-u+w}\prec\alpha_{w}$ and so $\mathcal{C}(P)$ contains the proper (black) SE-chain 
$$
\gamma_1,\gamma_2,\ldots,\gamma_{v-u+w},\alpha_w,\alpha_{w+1},\ldots,\alpha_{u'-1},
\gamma_{v'+1},\gamma_{v'+2},\ldots,\gamma_{k}
$$
of length $k$, a contradiction.
\end{itemize}
Since both cases lead to the contradiction, we conclude that $\tilde{P}$ can not contain any
black $k$-crossing. This proves the lemma. 
\end{proof}
\begin{lemma}\label{cyc_rotation}
The partition $P'$ obtained in step (4) of algorithm $\Phi$ has no black $k$-crossing. In other words,
the `cyclic rotation' does not create any black $k$-crossing. 
\end{lemma}
\begin{proof}
\newcommand\ii{\alpha}\newcommand\jj{\beta}
Recall that $\tilde{P}$ has no red node under a black $(k-1)$-crossing and that
it has no black $k$-crossing.
Let 
$$
\alpha_1=(i_1,j_1),\alpha_2=(i_2,j_2),\ldots,\red{\alpha_p=(i_p,j_p)},\ldots,\alpha_k=(i_{k},j_{k})
$$
be the greatest red $k$-crossing involved in the `cyclic rotation' to obtain $P'$ from $\tilde{P}$.
Let us consider the colored arc diagram $D^*$ which is obtained from $D(\tilde{P})$ by changing
the above red $k$-crossing into arcs 
$$
\beta_1=(i_1,j_2), \beta_2=(i_2,j_3), \ldots, \beta_{p-1}=(i_{p-1},j_p), \red{\beta_p=(i_p,j_1)}, \alpha_{p+1},\ldots,\alpha_k,
$$
where $(i_p,j_1)$ is recolored red.

To prove this lemma, it is sufficient to show that $D^*$ has no black $k$-crossing,
because if we let $P^*$ denote the partition whose $D(P^*)$ is obtained from $D^*$
by coloring $j_p$ black, $j_1$ red, and recoloring the nodes together with the arcs in the blocks
containing $j_p$ and $j_1$ accordingly, then any (new) red arc to the right of $\beta_p$ in $P^*$
will never be involved in a red $k$-crossing (since $D^*$ has no black $k$-crossing).
Moreover, it is not difficult to see that $P^*$ has no red $k$-crossing whose first arc is red.
Therefore, the next red arc involved in a red $k$-crossing of $P^*$ must occur strictly to
the left of $j_p$ and so after a finite number of steps of `cyclic rotation' we will obtain $P'$ which has
no black $k$-crossing. 

Now we proceed to show that $D^*$ has no black $k$-crossing. If not, suppose that $D^*$
contains a black $k$-crossing $\gamma=(\gamma_1,\gamma_2,\ldots,\gamma_k)$,
where $\gamma_i=(r_i,s_i)$. We have the following property similar to Fact~\ref{fact:1}. 
\begin{Fact}\label{fact:2}
For $1\leq u<u'\leq p-1$, if $\{\gamma_i:\beta_{u}\prec \gamma_i\prec \beta_{u'}\}\cap\{\beta_i: u<i<u'\}=\emptyset$, then
$$
|\{\gamma_i:\beta_{u}\prec \gamma_i\prec \beta_{u'}\}|\leq u'-u-1. 
$$
\end{Fact}
\begin{proof}
There is nothing to prove if $\{\gamma_i:\beta_{u}\prec \gamma_i\prec \beta_{u'}\}=\emptyset$.
So we can suppose that $\{\gamma_i:\beta_{u}\prec \gamma_i\prec \beta_{u'}\}=\{\gamma_i: v\leq i\leq v'\}$
for some $v\leq v'$. We claim that $v'-v\leq u'-u-2$, which is equivalent to our statement.
Otherwise, $v'-v\geq u'-u-1$ and the red $k$-crossing of $\tilde{P}$ 
$$
\alpha_1,\alpha_2,\ldots,\alpha_{u},\gamma_{v},\gamma_{v+1},\ldots,\gamma_{v+u'-u-1},\alpha_{u'+1},\alpha_{u'+2},\ldots,\alpha_{k}
$$
is greater than $\alpha_1,\alpha_2,\ldots,\alpha_{k}$ (since $\alpha_{u}\prec\gamma_{v}$,
$\gamma_{v+u'-u-1}\prec\alpha_{u'+1}$ and $s_v>j_{u+1}$), a contradiction. This completes the proof.
\end{proof}
It follows from Fact~\ref{fact:2} that one can choose the black $k$-crossing $\gamma$ of $D^*$ properly so that
$$
\{\beta_i: 1\leq i\leq k\}\cap\{\gamma_i: 1\leq i\leq k\}=\{\beta_{u}=\gamma_{v}\prec\beta_{u+1}=
\gamma_{v+1}\prec\cdots\prec\beta_{u'}=\gamma_{v'}\}
$$
for some $1\leq u\leq u'\leq p-1$ and $v'-v=u'-u$. We will distinguish two cases and use the $01$-filling
representation $\mathcal{C}$ of $\tilde{P}$ and $D^*$ (see Fig.~\ref{fig:cycroa}).
 \begin{figure}
\begin{center}
\begin{tikzpicture}[scale=1]
\draw[color=gray,fill=blue!10!white] (0,6) rectangle (2,4);
\draw[color=gray,fill=blue!10!white] (2,3.2) rectangle (3,5.2);
\draw[color=gray,fill=blue!10!white] (3,2.2) rectangle (4.5,4);
\draw[color=gray,fill=blue!10!white] (4.5,1) rectangle (6.5,3.2);
\draw[color=gray,fill=blue!10!white] (6.5,0) rectangle (9,2.2);
\draw[color=gray,fill=blue!10!white] (9,0) rectangle (10.5,1);
\draw[color=gray]
(0,6)--(14,6)--(14,-2) (0,5.2)--(14.8,5.2)--(14.8,-2)(0,4)--(3,4)(3,3.2)--(4.5,3.2)(4.5,2.2)--(6.5,2.2)(6.5,1)--(9,1)(9,0)--(10.5,0)--(10.5,5.2);
\node at (2, 5.2) {$\square$};\node at (2, 5.2) {$\bullet$};\node at (2.3, 5.5) {$\alpha_1$};
\node at (2,4.7) {\vector(0,-1){0.1}};
\node at (2,4) {$\square$}; \node at (2,4) {$\circ$}; \node at (1.7, 3.7) {$\beta_1$};
\node at (3,4) {$\square$};\node at (3,4) {$\bullet$};\node at (3.3,4.3) {$\alpha_2$};
\node at (3,3.7) {\vector(0,-1){0.1}};
\node at (3,3.2) {$\square$}; \node at (3,3.2){$\circ$}; \node at (2.7, 2.9) {$\beta_2$};
\node at (4.5,3.2) {$\square$}; \node at (4.5,3.2){$\bullet$}; \node at (4.8, 3.5) {$\alpha_3$};
\node at (4.5,2.8) {\vector(0,-1){0.1}};
\node at (4.5,2.2) {$\square$}; \node at (4.5,2.2){$\circ$};\node at (4.2, 1.9) {$\beta_3$};
\node at (6.5,2.2) {$\square$}; \node at (6.5,2.2){$\bullet$};\node at (6.8, 2.5) {$\alpha_4$};
\node at (6.5,1.7) {\vector(0,-1){0.1}};
\node at (6.5,1) {$\square$}; \node at (6.5,1){$\circ$};\node at (6.2, 0.7) {$\beta_4$};
\node at (7.7,1.7) {\Huge $\ddots$};
\node at (7.7,0.6) {\Huge $\ddots$};
\node at (9,1) {$\square$}; \node at (9,1){$\bullet$};\node at (9.5, 1.3) {$\alpha_{p-1}$};
\node at (9,0.6) {\vector(0,-1){0.1}};
\node at (9,0) {$\square$}; \node at (9,0){$\circ$};\node at (8.8, -0.3) {$\beta_{p-1}$};
\node at (10.5,0) {$\square$}; \node at (10.5,0){\red{$\bullet$}};\node at (11.5,0.3) {\red{$\alpha_p$} (red)};
\node at (10.5,2.6) {\vector(0,1){0.1}};
\node at (10.5,5.2) {$\square$}; \node at (10.5,5.2){\red{$\circ$}};\node at (11.5,4.9) {\red{$\beta_p$} (red)};
\node at (14,-1.5) {$\square$}; \node at (14,-1.5){$\bullet$};\node at (14.3,-1.2) {$\alpha_k$};
\node at (12,-0.7) {\Huge $\ddots$};
\end{tikzpicture}
\caption{The `cyclic rotation' on $01$-fillings of triangular shape. There is no $\gamma$ arc inside the shadow.}
\label{fig:cycroa}
\end{center}
\end{figure}
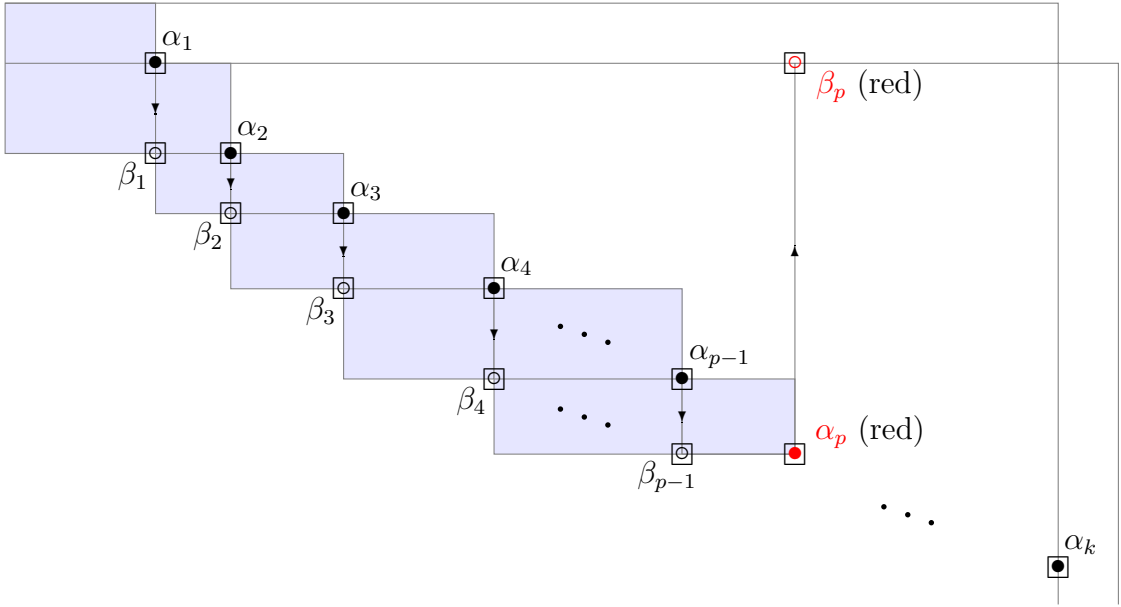
\begin{itemize}
\item {Case 1: there are arcs of $\gamma$ to the southwest of $\beta_1$ in $\mathcal{C}(D^*)$.}\\
We claim that
$$
|\{\gamma_i: \textrm{$j_{1}<s_i$ and $\gamma_i\prec \beta_{u}$}\}|<u.
$$
Otherwise, if $w=\min\{i:\textrm{$j_{1}<s_i$ and $\gamma_i\prec \beta_{u}$}\}$, then the red $k$-crossing of $\tilde{P}$ 
$$
\gamma_w,\gamma_{w+1},\ldots,\gamma_{w+u-1},\alpha_{u+1},\alpha_{u+2},\ldots,\alpha_k
$$
is greater than $\alpha_1,\alpha_2,\ldots,\alpha_{k}$, which is a contradiction.
Thus, the proper (black) SE-chain in $\mathcal{C}(\tilde{P})$ 
$$
\gamma_1,\gamma_2,\ldots,\gamma_{w-1},\alpha_1,\alpha_2,\ldots,\alpha_{u'},\gamma_{v'+1},\gamma_{v'+2},\ldots,\gamma_k
$$
has length at least $k$, which contradicts to the fact that $\tilde{P}$ has no black $k$-crossing. 
\item {Case 2: there is no arc of $\gamma$ to the southwest of $\beta_1$ in $\mathcal{C}(D^*)$.}\\
Let $w$, $2\leq w<u$, be the smallest index such that some arcs of $\gamma$ appear to
the southwest of $\beta_w$. If such an index does not exist, then simply set $w=u$.
Let $w'=\max\{i:\gamma_i\prec\beta_w\}$. It then follows from Fact~\ref{fact:2} that
$|\{\gamma_i:\beta_{w-1}\prec\gamma_i\prec\beta_u\}|\leq u-w$ and so the proper (black) SE-chain 
$$
\gamma_1,\gamma_2,\ldots,\gamma_{w'},\alpha_w,\alpha_{w+1},\ldots,\alpha_{u'}, \gamma_{v'+1},\gamma_{v'+2},\ldots,\gamma_k
$$
of $\mathcal{C}(\tilde{P})$ has length at least $k$, which is a contradiction again. 
\end{itemize}
We get the contradiction in both cases, which concludes that $D^*$ can not contain
any black $k$-crossing. The proof of the lemma is complete.
\end{proof}
\section{Final remarks, open problems}
\label{sec:final}
The main achievement of this paper is a delicate bijective proof of~\eqref{eq:NCNW} that 
extends several known special cases. At this point, we would like to pose the following 
two open problems.
\begin{Prob}
Is there any generating function approach to~\eqref{eq:NCNW}? 
\end{Prob}
It would be interesting to develop a generating tree approach to the usual and enhanced
$k$-noncrossing partitions analogous to~\cite{bemy} which may lead to a functional
equation proof of~\eqref{eq:NCNW}. 

There is a $q$-analogue of the $\gamma$-expansion of Narayana polynomials, which has 
interpretation in terms of statistics on noncrossing partitions, Dyck paths and
$321$-avoiding permutations (see~\cite{bp,lin0}). In view of the relationship in
Section~\ref{cross1} between the $\gamma$-expansion of Narayana polynomials and the
identity~\eqref{eq:t-CM}, it is natural to ask the following question. 
\begin{Prob}
Is there any nice $q$-analogue of~\eqref{eq:t-CM} which has a nice combinatorial 
interpretation?
\end{Prob}
Very recently, Gil and Tirrell~\cite{gt} found a bijective proof of~\eqref{eq:NCNW} when $t=1$.
Their bijection, though described directly using arc diagrams of partitions, is essentially
the same as our $f$ in Theorem~\ref{th:f}. 
\section*{Acknowledgement}
Authors thank Riccardo Biagioli and Jiang Zeng for helpful discussions. This work was supported
by the National Science Foundation of China grants 11871247 and 11501244,  by the Austrian
Science Foundation FWF, START grant Y463 and SFB grant F50, by the project of Qilu Young Scholars of Shandong University, and by the National Research
Foundation of Korea (NRF) grant funded by the Korea government (MSIT) (No. 2019R1F1A1062462).

\end{document}